\documentclass{amsart}
\usepackage[margin=1in]{geometry}
\usepackage[all]{xy}
\usepackage[normalem]{ulem}
\usepackage{verbatim}
\usepackage{color}
\usepackage{amsthm}
\usepackage{amssymb}
\usepackage[colorlinks=true]{hyperref}
\usepackage[multiple]{footmisc}

\numberwithin{equation}{section}

\newtheorem{theorem}[equation]{Theorem}
\newtheorem*{theorem*}{Theorem} \newtheorem{lemma}[equation]{Lemma}

\newtheorem*{conjecture*}{Mamma Conjecture}
\newtheorem*{conjecture1*}{Mamma Conjecture (revisited)}
\newtheorem{proposition}[equation]{Proposition}
\newtheorem{corollary}[equation]{Corollary}
\newtheorem*{corollary*}{Corollary}

\theoremstyle{remark}
\newtheorem{definition}[equation]{Definition}

\newtheorem{notation}[equation]{Notation}

\theoremstyle{remark}
\newtheorem{remark}[equation]{Remark}

\newcommand{\cA}{{\mathcal A}}
\newcommand{\cB}{{\mathcal B}}
\newcommand{\cC}{{\mathcal C}}
\newcommand{\cD}{{\mathcal D}}
\newcommand{\cE}{{\mathcal E}}
\newcommand{\cF}{{\mathcal F}}
\newcommand{\cG}{{\mathcal G}}

\newcommand{\cI}{{\mathcal I}}

\newcommand{\cO}{{\mathcal O}}

\newcommand{\cS}{{\mathcal S}}

\newcommand{\Spt}{\mathrm{Spt}}

\newcommand{\bbA}{\mathbb{A}}

\newcommand{\bbC}{\mathbb{C}}

\newcommand{\bbE}{\mathbb{E}}

\newcommand{\bbG}{\mathbb{G}}

\newcommand{\bbN}{\mathbb{N}}
\newcommand{\bbP}{\mathbb{P}}
\newcommand{\bbR}{\mathbb{R}}
\newcommand{\bbS}{\mathbb{S}}

\newcommand{\bbQ}{\mathbb{Q}}
\newcommand{\bbZ}{\mathbb{Z}}

\DeclareMathOperator{\id}{id}

\newcommand{\dgcat}{\mathrm{dgcat}} 

\newcommand{\bbK}{I\mspace{-6.mu}K}

\newcommand{\perf}{\mathrm{perf}}

\newcommand{\dg}{\mathrm{dg}}

\newcommand{\uHom}{\underline{\mathrm{Hom}}}
\newcommand{\Hom}{\mathrm{Hom}}

\newcommand{\dgHo}{\mathrm{H}^0}

\newcommand{\Ho}{\mathrm{Ho}}

\newcommand{\op}{\mathrm{op}}

\newcommand{\too}{\longrightarrow}

\let\oldmarginpar\marginpar
\def\marginpar#1{\oldmarginpar{\tiny #1}}

\begin{document}

\title[]{Motivic Atiyah-Segal completion theorem}
\author{Gon{\c c}alo~Tabuada and Michel Van den Bergh}
\address{Gon{\c c}alo Tabuada, Mathematics Institute, Zeeman Building, University of Warwick, Coventry CV4 7AL UK.}
\email{goncalo.tabuada@warwick.ac.uk}
\urladdr{https://homepages.warwick.ac.uk/staff/Goncalo.Tabuada/}

\thanks{G.~Tabuada was supported by the Huawei-IH\'ES research funds and M. Van den Bergh by the FWO grant G0D8616N: ``Hochschild cohomology and
  deformation theory of triangulated categories''.}

\address{Michel Van den Bergh, Departement WNI, Universiteit Hasselt, 3590 Diepenbeek, Belgium}
\email{michel.vandenbergh@uhasselt.be}
\email{michel.van.den.bergh@vub.be}
\urladdr{http://hardy.uhasselt.be/personal/vdbergh/Members/~michelid.html}
\date{\today}

\abstract{Let $T$ be a torus, $X$ a smooth quasi-compact separated scheme equipped with a $T$-action, and $[X/T]$ the
  associated quotient stack. Given any localizing $\bbA^1$-homotopy
  invariant of dg categories $E$ (homotopy $K$-theory, algebraic $K$-theory with coefficients, \'etale $K$-theory with coefficients, $l$-adic algebraic $K$-theory, $l$-adic \'etale $K$-theory, (real) semi-topological $K$-theory, topological $K$-theory, periodic cyclic homology, etc), we prove that the derived completion
  of $E([X/T])$ at the augmentation ideal $I$ of the representation ring $R(T)$
  of $T$ agrees with the classical Borel construction associated to the $T$-action on $X$. Moreover, for certain localizing $\bbA^1$-homotopy invariants, we extend this result to the case of a linearly reductive group scheme $G$. As a first application, we obtain an alternative proof of Krishna's completion theorem in algebraic $K$-theory, of Thomason's completion theorem in \'etale $K$-theory with coefficients, and also of Atiyah-Segal's completion theorem in topological $K$-theory (for those topological $M$-spaces $X^{\mathrm{an}}$ arising from analytification; $M$ is a(ny) maximal compact Lie subgroup of $G^{\mathrm{an}}$). These alternative proofs lead to a spectral enrichment of the corresponding completion theorems and also to the following improvements: in the case of Thomason's completion theorem the base field $k$ no longer needs to be separably closed, and in the case of Atiyah-Segal's completion theorem the topological $M$-space $X^{\mathrm{an}}$ no longer needs to be compact and the $M$-equivariant topological $K$-theory groups of $X^{\mathrm{an}}$ no longer need to be finitely generated over the representation ring $R(M)$. As a second application, we obtain new completion theorems in $l$-adic \'etale $K$-theory, in (real) semi-topological $K$-theory and also in periodic cyclic homology. As a third application, we obtain a purely algebraic description of the different equivariant cohomology groups in the literature (motivic, $l$-adic, (real) morphic, Betti, de Rham, etc). Finally, in two appendixes of
  independent interest, we extend a result of Weibel on homotopy
  $K$-theory from the realm of schemes to the broad setting of quotient stacks and establish some useful properties of (real) semi-topological $K$-theory.}}

\maketitle



\section{Statement of results}\label{sec:intro}
A {\em differential graded (=dg) category $\mathcal{A}$}, over a base field $k$, is a category enriched over complexes of $k$-vector spaces; consult Keller's survey \cite{Keller}. Every dg $k$-algebra $A$ gives naturally rise to a dg category. Another source of examples is provided by schemes (or, more generally, by algebraic stacks) since the category of perfect complexes $\perf(X)$ of every $k$-scheme $X$ (or algebraic stack) admits a canonical dg enhancement $\perf_\dg(X)$; consult \cite[\S4.6]{Keller} and \S\ref{sub:perfect}. Let us denote by $\dgcat(k)$ the category of (small) dg categories and by $\dgcat(k)_\infty$ the associated $\infty$-category of dg categories up to Morita equivalence; consult \cite[\S4.6]{Keller} for the notion of Morita equivalence and Lurie's monographs \cite{Lurie1,Lurie2} for the language of $\infty$-categories.

An $\infty$-functor $E\colon \dgcat(k)_\infty \to \mathcal{D}$, with values in a stable presentable $\infty$-category $\cD$, is called a {\em localizing $\bbA^1$-homotopy invariant} if it satisfies the following three conditions:
\begin{itemize}
\item[(C1)] it sends the short exact sequences of dg categories $0 \to \cA \to \cB \to \cC \to 0$ in the sense of Drinfeld/Keller (consult \cite{Drinfeld}\cite{Keller-Ann}\cite[\S4.6]{Keller}) to cofiber sequences $E(\cA) \to E(\cB) \to E(\cC)$.
\item[(C2)] it sends the canonical dg functors $\cA \to \cA[t]$, where $\cA[t]:=\cA\otimes k[t]$, to equivalences $E(\cA) \to E(\cA[t])$.
\item[(C3)] it preserve filtered colimits or it factors through an $\infty$-functor which preserve filtered colimits and which satisfies condition (C1).
\end{itemize}
Examples of localizing $\bbA^1$-homotopy invariants include homotopy $K$-theory, algebraic $K$-theory with coefficients, \'etale $K$-theory with coefficients, $l$-adic algebraic $K$-theory, $l$-adic \'etale $K$-theory, (real) semi-topological $K$-theory, topological $K$-theory, periodic cyclic homology, etc; consult \S\ref{sec:applications} for details. All these examples preserve filtered colimits except periodic cyclic homology; see \S\ref{sub:HP}.
\begin{notation}
In order to simplify the exposition, given an $\infty$-functor $E\colon \dgcat(k)_\infty \to \cD$ and a $k$-scheme $X$ (or algebraic stack), we will write $E(X)$ instead of $E(\perf_\dg(X))$.
\end{notation}
Let $T$ be a $k$-split torus of rank $r$, $X$ a smooth
quasi-compact separated $k$-scheme equipped with a $T$-action, and $[X/T]$ the associated quotient stack. As explained in \S\ref{sub:action}, given a localizing $\bbA^1$-homotopy invariant $E\colon \dgcat(k)_\infty \to \cD$, the $\bbE_\infty$-ring
$K([\bullet/T])$, i.e., the $T$-equivariant algebraic $K$-theory
spectrum of $\bullet:=\mathrm{Spec}(k)$, acts on $E([X/T])$. Note that
$\pi_0 K([\bullet/T])$ agrees with the representation ring
$R(T)\simeq \bbZ[\hat{T}]\simeq \bbZ[t_1^\pm, \ldots, t_r^\pm]$ of
$T$, where $\hat{T}$ stands for the group of characters. There are two very different constructions we can perform
(one algebraic and one geometric). On the algebraic side,
following Lurie \cite[\S4]{Lurie}, we can consider the derived
completion $E([X/T])^\wedge_I$ of the $K([\bullet/T])$-module
$E([X/T])$ at the augmentation ideal $I \subset R(T)$. On the geometric side, we can consider the Borel construction
\begin{equation}\label{eq:Borel}
E_T(X) := \mathrm{lim}_{j\geq 1} E([(X\times (\bbA^j\backslash \{0\})^r)/T])\,,
\end{equation}
where $T$ acts diagonally and the limit is taken over the inclusions $\bbA^j\backslash \{0\} = \bbA^j\backslash \{0\} \times \{0\} \hookrightarrow \bbA^{j+1}\backslash \{0\}$. Intuitively speaking, the tower of punctured affine spaces $\{\bbA^j\backslash\{0\}\}_{j \geq 1}$ plays the role of the ``contractible'' $\infty$-dimensional projective space $\bbP^\infty$ where the circle $\bbG_m$ acts freely (recall that $T=\bbG_m^r$). Note that since the $T$-action on $(\bbA^j\backslash\{0\})^r$ is free, the right-hand side of \eqref{eq:Borel} agrees with $\mathrm{lim}_{j\geq 1} E(X\times^T (\bbA^j\backslash \{0\})^r)$, where $X\times^T (\bbA^j\backslash \{0\})^r$ stands for the quotient of $X\times (\bbA^j\backslash \{0\})^r$ by $T$. This shows that the Borel construction may alternatively be performed entirely within the realm of schemes. Note also that since the projection maps $\mathrm{p}_j\colon X\times (\bbA^j\backslash \{0\})^r \to X, j \geq 1$, are $T$-equivariant they give rise to an induced morphism of $K([\bullet/T])$-modules:
\begin{equation}\label{eq:morphism}
E([X/T]) \too E_T(X)\,.
\end{equation} 

Finally, recall from Brion \cite[\S3]{Brion} (consult also \cite[\S5.2]{Krishna}) that $X$ is called {\em $T$-filtrable} if the closed subscheme $X^T$ of $T$-fixed points is smooth projective and if the following two conditions hold:
\begin{itemize}
\item[(i)] there exists an ordering $\coprod_{i=0}^m Z_i$ of the connected components of $X^T$ and also a filtration of $X$ by $T$-stable closed $k$-subschemes
\begin{equation}\label{eq:filtration}
\emptyset = X_{-1} \hookrightarrow X_0 \hookrightarrow \cdots \hookrightarrow X_i \hookrightarrow \cdots \hookrightarrow X_{m-1} \hookrightarrow X_m = X
\end{equation}
such that $Z_i \hookrightarrow W_i:=X_i\backslash X_{i-1}$ for every $0 \leq i \leq m$.
\item[(ii)] there exist $T$-equivariant vector bundles $\mathrm{q}_i\colon W_i \to Z_i$, with $0 \leq i \leq m$, for which the inclusions $Z_i \hookrightarrow W_i$ corresponds to the $0$-section embeddings.
\end{itemize}
Thanks to the work of Bialynicki-Birula \cite{BB} and Hesselink \cite{Hesselink}, every smooth {\em projective} $k$-scheme equipped with a $T$-action is $T$-filtrable.

\medskip

Our main result is the following:
\begin{theorem}[Motivic Atiyah-Segal completion theorem]\label{thm:completion}
Let $T$ be a $k$-split torus of rank $r$ and $X$ a smooth quasi-compact separated $k$-scheme equipped with a $T$-action which we assume to be $T$-filtrable. We assume moreover that $X$ is geometrically normal and geometrically reduced\footnote{Recall that every smooth $k$-scheme $X$ over a {\em perfect} base field $k$ is geometrically normal and geometrically reduced.}. Given a localizing $\bbA^1$-homotopy invariant $E\colon \dgcat(k)_\infty \to \cD$, the following hold:
\begin{itemize}
\item[(i)] The $K([\bullet/T])$-module $E_T(X)$ is $I$-complete in the sense of \cite[Def.~4.2.1]{Lurie}. Consequently, \eqref{eq:morphism} yields an induced morphism of $K([\bullet/T])$-modules $\theta\colon E([X/T])^\wedge_I \to E_T(X)$.
\item[(ii)] The preceding morphism $\theta$ is an equivalence.
\item[(iii)] Given an object $o \in \cD$ and an integer $n \in \mathbb{Z}$, let us write $\pi_{o,n}(-)$ for the functor $\Hom_{\Ho(\cD)}(o[n],-)$. Under these notations, we have induced isomorphisms
\begin{eqnarray}\label{eq:completion}
\pi_{o,n}(E([X/T])^\wedge_I) \simeq (\pi_{o,n}E([X/T]))^\wedge_I &&
\pi_{o,n}E_T(X)  \simeq \mathrm{lim}_{j\geq 1} \pi_{o,n} E(X\times^T(\bbA^j\backslash \{0\})^r)\,,
\end{eqnarray}
\end{itemize}
where the right-hand sides stand, respectively, for the classical completion and classical limit of abelian groups.
\end{theorem}
Theorem \ref{thm:completion} provides a striking connection between algebra and geometry. Intuitively speaking, it shows that the algebraic side given by the derived completion $E([X/T])^\wedge_I$ of the $K([\bullet/T])$-module $E([X/T])$ at the augmentation ideal $I$ ``matches perfectly'' with the geometric side given by the Borel construction $E_T(X)$. In particular, the Borel construction may be understood as a ``geometric completion'' construction. Such a striking connection between algebra and geometry goes back to the pioneering work of Atiyah-Segal \cite{AS} (they worked with compact topological spaces equipped with an action of a compact Lie group; consult Remark \ref{rk:AS}). Since Theorem \ref{thm:completion} holds for {\em every} localizing $\bbA^1$-homotopy invariant and is inspired by Atiyah-Segal's pioneering work, we decided to name it the ``motivic Atiyah-Segal completion theorem''.

\begin{remark}[Strategy of proof]
The proof of Theorem \ref{thm:completion} is divided into two main steps. In the first step we address the particular case where $T$ acts trivially on $X$. In the second step, making use of two key ingredients (namely, equivariant Gysin cofiber sequences and equivariant vector bundles; consult \S\ref{sec:Gysin}-\S\ref{sec:bundle}), we bootstrap the result from the particular case where $T$ acts trivially to the general case where $X$ is $T$-filtrable. Moreover, throughout the entire proof, an important (technical) role is played by the recent theory of noncommutative mixed motives; consult \S\ref{sub:NCmotives}.
\end{remark}
Now, let $G$ be a linearly reductive group $k$-scheme, $X$ a smooth separated $k$-scheme of finite type equipped with a $G$-action, and $[X/G]$ the associated quotient stack. Similarly to the case of a torus, we can perform two different constructions. On the algebraic side, we can consider the derived completion $E([X/G])^\wedge_{I_G}$ of the $K([\bullet/G])$-module $E([X/G])$ at the augmentation ideal $I_G \subset R(G)$. On the geometric side, given an admissible gadget $\{(V_j,U_j)\}_{j\geq 1}$ for $G$ in the sense of Morel-Voevodsky \cite[\S4.2]{MV} (consult also \cite[\S3.1]{Krishna}), we can consider the Borel construction $E_G(X):= \mathrm{lim}_{j \geq 1} E([(X\times U_j)/G])$, where $G$ acts diagonally and the limit is taken over the inclusions $U_j = U_j \times \{0\} \hookrightarrow U_{j+1}$. Since the projection maps $\mathrm{p}_j\colon X\times U_j \to X$, $j \geq 1$, are $G$-equivariant, they give rise to an induced morphism of $K([\bullet/G])$-modules:
\begin{equation}\label{eq:induced-G}
E([X/G]) \too E_G(X)\,.
\end{equation} 
Let $E\colon \dgcat(k)_\infty \to \cD$ be an $\infty$-functor with values in a stable presentable $\infty$-category. In addition to the above conditions (C1)-(C3), consider also the following extra condition:
\begin{itemize}
\item[(C4)] Let $\mathrm{i}\colon Z \hookrightarrow X$ a $G$-stable smooth
closed subscheme and $\mathrm{j}\colon U \hookrightarrow X$ the open
complement of $Z$. Under these notations, we have an induced cofiber sequence of $K([\bullet/G])$-modules:
\begin{equation}\label{eq:sequence}
E([Z/G]) \stackrel{\mathrm{i}_\ast}{\too} E([X/G]) \stackrel{\mathrm{j}^\ast}{\too} E([U/G])\,.
\end{equation}
\end{itemize}
\begin{remark}
\begin{itemize}
\item[(i)] Whenever the $\infty$-functor $E$ satisfies moreover the conditions (C1) and (C3), the cofiber sequence \eqref{eq:sequence} becomes a cofiber sequence of $K([\bullet/G])$-modules.
\item[(ii)] In the case where $G$ is a torus $T$, it is proved in Theorem \ref{thm:Gysin1} that condition (C4) follows from conditions (C1)-(C2)-(C3).
\end{itemize}
\end{remark}
As explained in \S\ref{sec:applications}, the majority of the aforementioned localizing $\bbA^1$-homotopy invariants satisfy the extra condition (C4). The next result may be understood as the motivic version of Thomason's classical ``reduction to a torus'' result (consult \cite[Thm.~1.13]{Thomason2} and \cite[\S4]{Thomason3}):
\begin{theorem}[Reduction to a torus]\label{thm:reduction}
Let $G$ be a linearly reductive group $k$-scheme containing a $k$-split maximal torus $T$, and $X$ a smooth separated $k$-scheme of finite type equipped with a $G$-action. Given a localizing $\bbA^1$-homotopy invariant $E\colon \dgcat(k)_\infty \to \cD$ satisfying the extra condition (C4), we have the following commutative diagram of $K([\bullet/G])$-modules with $\mathrm{ind}\circ \mathrm{res} =\id$:
\begin{eqnarray}\label{eq:diagram-key}
\xymatrix{
E([X/G]) \ar[d]_-{\eqref{eq:induced-G}} \ar[r]^-{\mathrm{res}} & E([X/T]) \ar[r]^-{\mathrm{ind}} \ar[d]_-{\eqref{eq:morphism}} & E([X/G]) \ar[d]_-{\eqref{eq:induced-G}} \\
E_G(X) \ar[r]_-{\mathrm{res}} & E_T(X) \ar[r]_-{\mathrm{ind}} & E_G(X)\,.
} \,.
\end{eqnarray}
\end{theorem}
\begin{corollary}[Motivic Atiyah-Segal completion theorem]\label{cor:main} 
Let $G$ be a linearly reductive group scheme, $T$ a $k$-split maximal torus of $G$, and $X$ a smooth separated $k$-scheme of finite type equipped with a $G$-action which we assume to be $T$-filtrable. We assume moreover that $X$ is geometrically normal and geometrically reduced. Given a localizing $\bbA^1$-homotopy invariant $E\colon \dgcat(k)_\infty \to \cD$ satisfying the extra condition (C4), Theorem \ref{thm:completion} holds similarly with $T$ replaced by $G$.
\end{corollary}
\begin{proof}
It is well-known that the $I_G$-adic and the $I$-adic topologies on the representation ring $R(T)$ coincide; consult \cite[Cor.~6.1]{EG}\cite[Cor.~3.9]{Segal}. Hence, the proof follows from the combination of Theorem \ref{thm:completion} with the commutative diagram \eqref{eq:diagram-key}. 
\end{proof}
\section{Applications}\label{sec:applications}

In this section we describe several applications of the motivic Atiyah-Segal completion theorem.

\subsection{Homotopy $K$-theory}\label{sub:K-theory}
Consider the simplicial $k$-algebra $\Delta_m:=k[t_0, \ldots, t_m]/\langle\sum_{i=0}^m t_i-1\rangle, m \geq 0$, equipped with the following faces and degenerancies:
\begin{eqnarray*}
d_q(t_i):=\begin{cases} t_i & \text{if}\quad i<q \\ 0 & \text{if}\quad i=q \\ t_{i-1}& \text{if}\quad i>q\end{cases} &&  s_q(t_i):=\begin{cases} t_i & \text{if}\quad i<q \\ t_i + t_{i+1} & \text{if}\quad i=q \\ t_{i+1}& \text{if}\quad i>q\end{cases} \,.
\end{eqnarray*}
Following Weibel \cite{Weibel}, {\em homotopy $K$-theory} is defined as follows
\begin{eqnarray}\label{eq:KH}
KH(-)\colon \dgcat(k)_\infty \too \mathrm{Spt}_\infty && \cA \mapsto \mathrm{colim}_{m} \bbK(\cA\otimes \Delta_m)\,,
\end{eqnarray}
where $\bbK$ stands for nonconnective algebraic $K$-theory. The $\infty$-functor \eqref{eq:KH}, with values in the $\infty$-category of spectra, is a localizing $\bbA^1$-homotopy invariant: condition (C1) follows from Thomason-Trobaugh's work \cite[\S5]{TT} and from the fact that the $\infty$-functor $-\otimes \Delta_m$ preserve short exact sequences of dg categories\footnote{As proved by Drinfeld in \cite[Prop.~1.6.3]{Drinfeld}, given any dg category $\cB$, the associated $\infty$-functor $-\otimes \cB$ preserve short exact sequences of dg categories.}, condition (C2) follows from \cite[Prop.~5.2]{A1homotopy}, and condition (C3) follows from the fact that the $\infty$-functors $-\otimes \Delta_m$ and $\bbK(-)$ preserve filtered colimits; consult \cite[Example~8.20]{book}. The $\infty$-functor \eqref{eq:KH} satisfies moreover the extra condition (C4): this follows from the combination of Theorem \ref{thm:homotopy} with Thomason's work \cite[Thms.~2.7 and 5.7]{ThomasonActions}. Therefore, Theorem \ref{thm:completion} and Corollary \ref{cor:main} applied to $E=KH(-)$ and to the sphere spectrum $o=\bbS$ yield the following equivalence of $K([\bullet/G])$-modules
\begin{eqnarray}\label{eq:K-theory33}
\theta\colon KH([X/G])^\wedge_{I_G} & \too & KH_G(X)
\end{eqnarray}
and the following isomorphisms of $R(G)$-modules:
\begin{eqnarray}\label{eq:K-theory44}
\theta_\ast\colon (KH_\ast([X/G]))^\wedge_{I_G} & \stackrel{\simeq}{\too} & \mathrm{lim}_{j\geq 1} KH_\ast(X\times^G U_j)\,.
\end{eqnarray}
%
Since the $k$-scheme $X$ is smooth, resp. the $k$-schemes $X\times^G U_j$ are smooth, it follows from Theorem \ref{thm:homotopy}, resp. from \cite[Prop.~6.10]{Weibel}, that in \eqref{eq:K-theory33}-\eqref{eq:K-theory44} we can replace $KH(-)$ by $\bbK(-)$. We obtain in this way the following equivalence and isomorphisms:
\begin{eqnarray*}
\theta\colon \bbK([X/G])^\wedge_{I_G} \too  \bbK_G(X) && 
\theta_\ast\colon (\bbK_\ast([X/G]))^\wedge_{I_G} \stackrel{\simeq}{\too} \mathrm{lim}_{j\geq 1} \bbK_\ast(X\times^G U_j)\,.
\end{eqnarray*}
To the best of the authors' knowledge, the latter equivalence $\theta$ is new in the literature. In what concerns the latter isomorphisms $\theta_\ast$, they were originally established by Krishna in \cite[Thm.~1.2]{Krishna} using different arguments. For example, in the case of a torus $T$, Krishna made essential use of the classical localization long exact sequence in equivariant $G$-theory (established by Thomason in \cite{ThomasonActions}). It is not known if such a long exact sequence holds for every localizing $\bbA^1$-homotopy invariant because its proof is based on Quillen's d\'evissage theorem, which is a result very specific to the $G$-theory of abelian categories. Our proof of Theorem \ref{thm:completion} circumvents this difficulty by using instead the $T$-equivariant Gysin cofiber sequences.    
\begin{remark}[$\bbQ$-coefficients]
By composing \eqref{eq:KH} with the $\bbQ$-linearization $\infty$-functor $(-)\otimes \bbQ$, we obtain the localizing $\bbA^1$-homotopy invariant $KH(-)\otimes \bbQ$. As above, this leads in particular to the isomorphisms:
\begin{eqnarray}\label{eq:isos-theta}
\theta_n \colon (\bbK_n([X/G])_\bbQ)^\wedge_{I_G} \stackrel{\simeq}{\too} \mathrm{lim}_{j \geq 1} \bbK_n(X\times^G U_j)_\bbQ && n \in \bbZ\,.
\end{eqnarray}
Since $X\times^G U_j$ is smooth, the associated motivic spectral sequence degenerates rationally; consult Grayson's survey \cite{Grayson}. This yields an isomorphism between $\bbK_n(X\times^G U_j)_\bbQ$ and $\bigoplus_{i \in \bbZ} H^{2i-n}_{\mathrm{mot}}(X\times^G U_j; \bbQ(i))$, where $H^\ast_{\mathrm{mot}}(-;\bbQ(\ast))$ stands for motivic cohomology. Consequently, by definition of the $G$-equivariant motivic cohomology groups $H^{2i-n}_{G, \mathrm{mot}}(X;\bbQ(i))$, we obtain from \eqref{eq:isos-theta} the following isomorphisms:
\begin{eqnarray}\label{eq:iso-motivic}
(\bbK_n([X/G])_\bbQ)^\wedge_{I_G} \simeq \prod_{i \in \bbZ} H^{2i-n}_{G, \mathrm{mot}}(X;\bbQ(i)) && n \in \bbZ\,.
\end{eqnarray}
To the best of the authors' knowledge, the isomorphisms \eqref{eq:iso-motivic} are new in the literature; consult the related works \cite{EG,RR-K}. Intuitively speaking, they show that the product of the $G$-equivariant motivic cohomology groups of $X$ admits a purely algebraic description given by the (classical) completion at the augmentation ideal $I_G$ of the $\bbQ$-linearized algebraic $K$-theory groups of the quotient stack $[X/G]$.
\end{remark}
\subsection{Algebraic $K$-theory with coefficients}\label{sub:coefficients}
Let $l^\nu$ be a prime power. Following Browder \cite{Browder}, {\em algebraic $K$-theory with $\bbZ/l^\nu$-coefficients} is defined as follows
\begin{eqnarray}\label{eq:coefficients}
\bbK(-;\bbZ/l^\nu)\colon \dgcat(k)_\infty \too \mathrm{Spt}_\infty && \cA \mapsto \bbK(\cA)\wedge \bbS/l^\nu\,,
\end{eqnarray}
where $\bbS/l^\nu$ stands for the mod-$l^\nu$ Moore spectrum. Note that we have the following short exact sequences:
$$ 0 \too \bbK_\ast(\cA) \otimes_\bbZ \bbZ/l^\nu \too \bbK_\ast(\cA;\bbZ/l^\nu) \too \{l^\nu\text{-}\text{torsion}\,\,\text{in}\,\,\bbK_{\ast-1}(\cA)\} \too 0\,.$$
Assume that $1/l\in k$. Under this assumption, the $\infty$-functor \eqref{eq:coefficients} is a localizing $\bbA^1$-homotopy invariant: condition (C1) follows from the fact that $\bbK(-)$ satisfies condition (C1) and that $-\wedge \bbS/l^\nu$ preserve cofiber sequences, condition (C2) follows from \cite[Thm.~1.2]{Tab}, and condition (C3) follows from the fact that $\bbK(-)$ and $-\wedge \bbS/l^\nu$ preserve filtered colimits; consult \cite[Example~8.21]{book}. The $\infty$-functor \eqref{eq:coefficients} satisfies moreover the extra condition (C4): this follows from Thomason's work \cite[Thms.~2.7 and 5.7]{ThomasonActions} and from the fact that $-\wedge \bbS/l^\nu$ preserve cofiber sequences. Therefore, Theorem \ref{thm:completion} and Corollary \ref{cor:main} applied to $E=\bbK(-;\bbZ/l^\nu)$ and to $o=\bbS$ yields the following equivalence and isomorphisms:
\begin{eqnarray*}
\theta\colon \bbK([X/G];\bbZ/l^\nu)^\wedge_{I_G} \too  \bbK_G(X;\bbZ/l^\nu) &&
\theta_\ast\colon (\bbK_\ast([X/G];\bbZ/l^\nu))^\wedge_{I_G} \stackrel{\simeq}{\too}  \mathrm{lim}_{j\geq 1} \bbK_\ast(X\times^G U_j;\bbZ/l^\nu)\,.
\end{eqnarray*}
To the best of the authors' knowledge, both $\theta$ as well as $\theta_\ast$ are new in the literature.
\subsection{\'Etale $K$-theory with coefficients}\label{sub:etale1}
Let $l^\nu$ be a prime power. Following Thomason \cite{Thomason-etale}, {\em \'etale $K$-theory with $\bbZ/l^\nu$-coefficients} is defined as follows
\begin{eqnarray}\label{eq:etale1}
K^{\mathrm{et}}(-;\bbZ/l^\nu)\colon \dgcat(k)_\infty \too \mathrm{Spt}_\infty && \cA \mapsto L_{K(1)}\bbK(\cA;\bbZ/l^\nu)\,,
\end{eqnarray}
where $L_{K(1)}(-)$ stands for the left Bousfield localization with respect to the first Morava $K$-theory $K(1)$. Assume that $1/l \in k$. Under this assumption, the $\infty$-functor \eqref{eq:etale1} is a localizing $\bbA^1$-homotopy invariant: conditions (C1)-(C3) follow from \S\ref{sub:coefficients} and from the fact that $L_{K(1)}(-)$ preserve cofiber sequences and filtered colimits. The $\infty$-functor \eqref{eq:etale1} satisfies moreover the extra condition (C4): this follows from \S\ref{sub:coefficients} and from the fact that $L_{K(1)}(-)$ preserve cofiber sequences. Therefore, Theorem \ref{thm:completion} and Corollary \ref{cor:main} applied to $E=K^{\mathrm{et}}(-;\bbZ/l^\nu)$ and to $o=\bbS$ yields the following equivalence and isomorphisms:
\begin{eqnarray*}
\theta\colon K^{\mathrm{et}}([X/G];\bbZ/l^\nu)^\wedge_{I_G}  \too K^{\mathrm{et}}_G(X; \bbZ/l^\nu)  &&
\theta_\ast\colon (K^{\mathrm{et}}_\ast([X/G];\bbZ/l^\nu))^\wedge_{I_G} \stackrel{\simeq}{\too} \mathrm{lim}_{j\geq 1} K^{\mathrm{et}}_\ast(X\times^G U_j; \bbZ/l^\nu)\,.
\end{eqnarray*}
To the best of the authors' knowledge, the equivalence $\theta$ is new in the literature. In what concerns the isomorphisms $\theta_\ast$, they were originally established by Thomason\footnote{Thomason used instead the classical simplicial model $\{K^{\mathrm{et}}(X\times G^m; \bbZ/l^\nu)\}_{m \geq 0}$ for the Borel construction.} in \cite[Cor.~3.4]{Thomason} using different arguments and under the additional assumption that $k$ is separably closed. For example, similarly to Krishna's proof, in the case of a torus $T$, Thomason made essential use of the classical localization long exact sequence in equivariant $G$-theory (note that Thomason's work precedes Krishna's work). Moreover, under the aforementioned additional assumptions, he used some deep computations in \'etale cohomology in order to prove that the $R(T)$-modules $K^{\mathrm{et}}_\ast([X/T]; \bbZ/l^\nu)$ are finitely generated. Thanks to the Artin-Rees lemma, this finite generation result enabled him to use in an essential way the exactness of the classical completion functor. Our proof of Theorem \ref{thm:completion} circumvents both these difficulties. On the one hand, instead of the classical localization long exact sequence in equivariant $G$-theory, we use the $T$-equivariant Gysin cofiber sequences. On the other hand, instead of the classical completion functor, we use the derived completion functor\footnote{Consult also the work of Carlsson-Joshua \cite{CJ}, where they used a certain derived completion functor developed in \cite{Carlsson}.} which is always exact. In fact, Thomason already suggested in \cite[page 795]{Thomason} that if one could construct a completion functor at the deep level of spectra (in contrast to the classical completion functor at the superficial level of homotopy groups of spectra), then one would likely be able to extend his result to other base fields.
\begin{remark}[Analytification]\label{rk:analytic}
Let $k=\bbC$. In this case, we can consider the analytic topological space $X^{\mathrm{an}}$ associated to $X$, i.e., the set of complex points $X(\bbC)$ equipped with the usual analytic topology, and also the Lie group $G^{\mathrm{an}}$. Let $M$ be a(ny) maximal compact Lie subgroup of $G^{\mathrm{an}}$. Under these notations, Thomason proved in \cite[\S5]{Thomason2}\cite[\S2]{Thomason3} that the analytification functor (sending ($G$-equivariant) vector bundles over $X$ to ($M$-equivariant) vector bundles over $X^{\mathrm{an}}$) gives rise to the following isomorphisms
\begin{eqnarray*}\label{eq:cohomology-new}
K_\ast^{\mathrm{et}}([X/G];\bbZ/l^\nu) \stackrel{\simeq}{\too} K_{M, \mathrm{Segal}}^{-\ast}(X^{\mathrm{an}};\bbZ/l^\nu) &&
\mathrm{lim}_{j\geq 1} K^{\mathrm{et}}_\ast(X\times^G U_j; \bbZ/l^\nu) \stackrel{\simeq}{\too} K_{M, \mathrm{Borel}}^{-\ast}(X^{\mathrm{an}};\bbZ/l^\nu)\,,
\end{eqnarray*} 
where $K_{M, \mathrm{Segal}}^\ast(-;\bbZ/l^\nu)$ stands for Segal's $M$-equivariant topological $K$-theory with $\bbZ/l^\nu$-coefficients and $K_{M, \mathrm{Borel}}^\ast(-;\bbZ/l^\nu)$ for Borel's $M$-equivariant topological $K$-cohomology with $\bbZ/l^\nu$-coefficients.
\end{remark}
\subsection{$l$-adic algebraic $K$-theory}\label{sub:completed}
Let $l$ be a prime number. Following Thomason \cite{Thomason-etale}, {\em $l$-adic algebraic $K$-theory} is defined as follows:
\begin{eqnarray}\label{eq:completed}
\bbK(-)_{\hat{l}} \colon \dgcat(k)_\infty \too \mathrm{Spt}_\infty && \cA \mapsto \mathrm{lim}_\nu \bbK(\cA;\bbZ/l^\nu)\,.
\end{eqnarray}
Note that we have the following Milnor's short exact sequences:
$$ 0 \too \mathrm{lim}^1_\nu \bbK_{\ast+1}(\cA;\bbZ/l^\nu) \too \pi_\ast(\bbK(\cA)_{\hat{l}}) \too \mathrm{lim}_\nu \bbK_\ast(\cA;\bbZ/l^\nu) \too 0\,.$$
Assume that $1/l\in k$. Under this assumption, the $\infty$-functor \eqref{eq:completed} is a localizing $\bbA^1$-homotopy invariant: condition (C1) follows from \S\ref{sub:coefficients} and from the fact that $(-)_{\hat{l}}$ preserve (co)fiber sequences, condition (C2) follows from \S\ref{sub:coefficients}, and condition (C3) follows from the fact that $\bbK(-)_{\hat{l}}$ factors through $\bbK(-)$ (which preserve filtered colimits and satisfies condition (C1)). The $\infty$-functor \eqref{eq:completed} satisfies moreover the extra condition (C4): this follows from \S\ref{sub:coefficients} and from the fact that $(-)_{\widehat{l}}$ preserve (co)fiber sequences. Therefore, Theorem \ref{thm:completion} and Corollary \ref{cor:main} applied to $E=\bbK(-)_{\hat{l}}$ and to $o=\bbS$ yields the following equivalence and isomorphisms:
\begin{eqnarray*}
\theta\colon (\bbK([X/G])_{\hat{l}})^\wedge_{I_G} \too  \bbK_G(X)_{\hat{l}} && 
\theta_\ast\colon (\pi_\ast(\bbK([X/G])_{\hat{l}}))^\wedge_{I_G} \stackrel{\simeq}{\too} \mathrm{lim}_{j\geq 1} \pi_\ast(\bbK(X\times^G U_j)_{\hat{l}})\,.
\end{eqnarray*}
To the best of the authors' knowledge, both $\theta$ as well as $\theta_\ast$ are new in the literature. 
\subsection{$l$-adic \'etale $K$-theory}\label{sub:etale2}
Let $l$ be a prime number. Following Thomason \cite{Thomason-etale}, {\em $l$-adic \'etale $K$-theory} is defined as follows:
\begin{eqnarray}\label{eq:etale11}
K^{\mathrm{et}}(-)_{\hat{l}}\colon \dgcat(k)_\infty \too \mathrm{Spt}_\infty && \cA \mapsto \mathrm{lim}_\nu K^{\mathrm{et}}(\cA;\bbZ/l^\nu)\,.
\end{eqnarray}
Assume that $1/l\in k$. Under this assumption, the $\infty$-functor \eqref{eq:etale11} is a localizing $\bbA^1$-homotopy invariant: condition (C1) follows from \S\ref{sub:etale1} and from the fact that $(-)_{\hat{l}}$ preserve (co)fiber sequences, condition (C2) follows from \S\ref{sub:etale1}, and condition (C3) follows from the fact that $K^{\mathrm{et}}(-)_{\hat{l}}$ factors through $\bbK(-)$ (which preserve filtered colimits and satisfies condition (C1)). The $\infty$-functor \eqref{eq:etale11} satisfies moreover the extra condition (C4): this follows from \S\ref{sub:etale1} and from the fact that $(-)_{\widehat{l}}$ preserve (co)fiber sequences. Therefore, Theorem \ref{thm:completion} and Corollary \ref{cor:main} applied to $E=K^{\mathrm{et}}(-)_{\hat{l}}$ and to $o=\bbS$ yields the equivalence and isomorphisms:
\begin{eqnarray*}
\theta\colon (K^{\mathrm{et}}([X/G])_{\hat{l}})^\wedge_{I_G} \too K^{\mathrm{et}}_G(X)_{\hat{l}} &&
\theta_\ast\colon (\pi_\ast(K^{\mathrm{et}}([X/G])_{\hat{l}}))^\wedge_{I_G} \stackrel{\simeq}{\too} \mathrm{lim}_{j\geq 1} \pi_\ast(K^{\mathrm{et}}(X\times^G U_j)_{\hat{l}})\,.
\end{eqnarray*}
To the best of the authors' knowledge, both $\theta$ as well as $\theta_\ast$ are new in the literature. 
\begin{remark}[$\bbQ_l$-coefficients]
By composing \eqref{eq:etale11} with the $\bbZ[1/l]$-linearization $\infty$-functor $(-)[1/l]$, we obtain the localizing $\bbA^1$-homotopy invariant $K^{\mathrm{et}}(-)_{\hat{l}}[1/l]$. As above, this leads to the isomorphisms:
\begin{eqnarray}\label{eq:isos-l-adic-last}
\theta_n\colon (\pi_n(K^{\mathrm{et}}([X/G])_{\hat{l}})_{1/l})^\wedge_{I_G} \stackrel{\simeq}{\too} \mathrm{lim}_{j\geq 1} \pi_n(K^{\mathrm{et}}(X\times^G U_j)_{\hat{l}})_{1/l} && n \in \bbZ\label{eq:K-theory13}\,.
\end{eqnarray}
Let us assume that $X\times^G U_j$ is moreover of finite Krull dimension and that all its residue fields have finite and uniformly bounded mod-$l$ virtual \'etale cohomological dimension. 
Under these additional assumptions, Thomason's \'etale descent spectral sequence degenerates rationally; consult Thomason \cite[Thm.~4.1]{Thomason-etale}, Soul\'e \cite[\S3.3.2]{Soule} and Rosenschon-{\O}stvaer \cite{RO,RO1}. Consequently, in the case where $k$ contains all the $l^{\mathrm{th}}$ power roots of unity, we obtain an isomorphism between $\pi_n(K^{\mathrm{et}}(X\times^G U_j)_{\hat{l}})_{1/l}$ and $\bigoplus_{i\,\mathrm{even}} H^i_{l\text{-}\mathrm{adic}}(X\times^G U_j)$, resp. $\bigoplus_{i\,\mathrm{odd}} H^i_{l\text{-}\mathrm{adic}}(X\times^G U_j)$, when $n$ is even, resp. odd, where  $H^\ast_{l\text{-}\mathrm{adic}}(-):=(\mathrm{lim}_\nu H^\ast_{\mathrm{et}}(-;\bbZ/l^\nu))\otimes_{\bbZ_l} \bbQ_l$ stands for $l$-adic cohomology. Consequently, by definition of the $G$-equivariant $l$-adic cohomology groups $H^i_{G, l\text{-}\mathrm{adic}}(X)$, we obtain from \eqref{eq:isos-l-adic-last} the following isomorphisms:
\begin{equation}\label{eq:cohomology-l-adic}
(\pi_n(K^{\mathrm{et}}([X/G])_{\hat{l}})_{1/l})^\wedge_{I_G} \simeq  \begin{cases} \prod_{i\,\mathrm{even}} H^i_{G, l\text{-}\mathrm{adic}}(X) & \mathrm{if}\,\, n \,\,\mathrm{even} \\
 \prod_{i\,\mathrm{odd}} H^i_{G, l\text{-}\mathrm{adic}}(X) & \mathrm{if}\,\, n \,\,\mathrm{odd} \,.
\end{cases}
\end{equation}
To the best of the authors' knowledge, the isomorphisms \eqref{eq:cohomology-l-adic} are also new in the literature.
\end{remark}
\subsection{Semi-topological $K$-theory}\label{sub:semi}
Let $k=\bbC$. Consider the standard topological simplex $\Delta^m_{\mathrm{top}}$, with $m \geq 0$, and the category $\Delta^{m, \downarrow}_{\mathrm{top}}$ whose the objects are the pairs $(V,f_V)$, where $V$ is an affine $\bbC$-scheme of finite type and $f_V\colon \Delta^m_{\mathrm{top}} \to V^{\mathrm{an}}$ is a continuous map of topological spaces, and whose morphisms $(V,f_V) \to (W,f_W)$ are the maps of $\bbC$-schemes $g\colon W \to V$ such that $g^{\mathrm{an}}\circ f_W=f_V$. Following Friedlander-Walker \cite{FW1} (based on a suggestion of Voevodsky), {\em semi-topological $K$-theory} is defined as follows (consult \cite{Blanc}):
\begin{eqnarray}\label{eq:semi}
K^{\mathrm{st}}(-)\colon \dgcat(\bbC)_\infty \too \mathrm{Spt}_\infty && \cA \mapsto \mathrm{colim}_m K^{\mathrm{st}}(\cA)_m \,,
\end{eqnarray}
where $K^{\mathrm{st}}(\cA)_m$ stands for the spectrum $\mathrm{colim}_{(V,f_V) \in \Delta^{m, \downarrow}_{\mathrm{top}}} \,\bbK(\perf_\dg(V) \otimes \cA)$. As proved in Theorem \ref{prop:semi}, the $\infty$-functor \eqref{eq:semi} is a localizing $\bbA^1$-homotopy invariant which satisfies the extra condition (C4). Therefore, Theorem \ref{thm:completion} and Corollary \ref{cor:main} applied to $E=K^{\mathrm{st}}(-)$ and to $o=\bbS$ yields the equivalence and isomorphisms:
\begin{eqnarray*}
\theta\colon K^{\mathrm{st}}([X/G])^\wedge_{I_G} \too K^{\mathrm{st}}_G(X) && \theta_\ast\colon (K^{\mathrm{st}}_\ast([X/G]))^\wedge_{I_G} \stackrel{\simeq}{\too} \mathrm{lim}_{j \geq 1} K^{\mathrm{st}}_\ast(X\times^G U_j)\,.
\end{eqnarray*}
To the best of the authors' knowledge, both $\theta$ as well as $\theta_\ast$ are new in the literature.
\begin{remark}[$\bbQ$-coefficients]\label{rk:semi}
By composing \eqref{eq:semi} with the $\bbQ$-linearization $\infty$-functor $(-)\otimes \bbQ$, we obtain the localizing $\bbA^1$-homotopy invariant $K^{\mathrm{st}}(-)\otimes\bbQ$. As above, this leads to the isomorphisms:
\begin{eqnarray}\label{eq:isos-semi}
\theta_n\colon (K_n^{\mathrm{st}}([X/G])_\bbQ)^\wedge_{I_G} \stackrel{\simeq}{\too} \mathrm{lim}_{j \geq 1} K_n^{\mathrm{st}}(X\times^G U_j)_\bbQ && n \in \bbZ\,.
\end{eqnarray}
Let us assume that $X\times^G U_j$ is quasi-projective. In this case, the semi-topological spectral sequence degenerates rationally; consult Friedlander-Haesemeyer-Walker \cite[Thm.~4.2]{FW2}. This yields an isomorphism between $K_n^{\mathrm{st}}(X\times^G U_j)_\bbQ$ and $\bigoplus_{i \in \bbZ} L^i H^{2i-n}(X\times^G U_j)_\bbQ$, where $L^\ast H^\ast(-)_\bbQ$ stands for morphic cohomology in the sense of Friedlander-Lawson \cite{FL}. Consequently, by definition of the $G$-equivariant morphic cohomology groups $L^iH^{2i-n}_G(X)_\bbQ$, we obtain from \eqref{eq:isos-semi} the following isomorphisms:
\begin{eqnarray}\label{eq:iso-motivic-semi-last}
(K_n^{\mathrm{st}}([X/G])_\bbQ)^\wedge_{I_G} \simeq \prod_{i \in \bbZ} L^i H_G^{2i-n}(X)_\bbQ && n \in \bbZ\,.
\end{eqnarray}
To the best of the authors' knowledge, the isomorphisms \eqref{eq:iso-motivic-semi-last} are also new in the literature.
\end{remark}
\begin{remark}[Real semi-topological $K$-theory]\label{rk:real}
Let $k=\bbR$. Thanks to the work of Friedlander-Walker \cite{FW3}, all the above holds {\em mutatis mutandis} with $\bbC$ replaced by $\bbR$. In particular, we have isomorphisms
\begin{eqnarray}\label{eq:iso-motivic-semi-last1}
(K\!R_n^{\,\mathrm{st}}([X/G])_\bbQ)^\wedge_{I_G} \simeq \prod_{i \in \bbZ} L^i H\!R_G^{\,2i-n}(X)_\bbQ && n \in \bbZ\,,
\end{eqnarray}
where $K\!R^{\,\mathrm{st}}(-)$ stands for real semi-topological $K$-theory and $L^\ast H\!R^\ast(-)$ for real morphic cohomology.
\end{remark}
\subsection{Topological $K$-theory}\label{sub:top}
Let $k=\bbC$. Following Friedlander-Walker \cite{FW4}, {\em topological $K$-theory} can be recovered from semi-topological $K$-theory as follows (consult \cite{Blanc})
\begin{eqnarray}\label{eq:topological}
K^{\mathrm{top}}(-)\colon \dgcat(\bbC)_\infty \too \mathrm{Spt}_\infty&& \cA \mapsto K^{\mathrm{st}}(\cA) \wedge_{\mathrm{bu}}\mathrm{BU}\,,
\end{eqnarray}
where $\mathrm{bu}=K^{\mathrm{st}}(\bbC)$ stands for the connective cover of the classical topological $K$-theory spectrum $\mathrm{BU}$. The $\infty$-functor \eqref{eq:topological} is a localizing $\bbA^1$-homotopy invariant: conditions (C1)-(C3) follows from \S\ref{sub:semi} and from the fact that $-\wedge_{\mathrm{bu}}\mathrm{BU}$ preserve cofiber sequences and filtered colimits. The $\infty$-functor \eqref{eq:topological} satisfies moreover the extra condition (C4): this follows from \S\ref{sub:semi} and from the fact that $-\wedge_{\mathrm{bu}}\mathrm{BU}$ preserve cofiber sequences. Therefore, Theorem \ref{thm:completion} and Corollary \ref{cor:main} applied to $E=K^{\mathrm{top}}(-)$ and to $o=\bbS$ yields the following equivalence and isomorphisms:
\begin{eqnarray}\label{eq:theta}
\theta\colon K^{\mathrm{top}}([X/G])^\wedge_{I_G} \too K^{\mathrm{top}}_G(X) && \theta_\ast\colon (K^{\mathrm{top}}_\ast([X/G]))^\wedge_{I_G} \stackrel{\simeq}{\too} \mathrm{lim}_{j \geq 1} K^{\mathrm{top}}_\ast(X\times^G U_j)\,.
\end{eqnarray}
To the best of the authors' knowledge, both $\theta$ as well as $\theta_\ast$ are new in the literature. 
\begin{remark}[Atiyah-Segal's completion theorem in topological $K$-theory]\label{rk:AS}
Note that similarly to Remark \ref{rk:analytic}, the right-hand side of \eqref{eq:theta} may be re-written as follows
\begin{equation}\label{eq:AS1}
\theta_\ast\colon (K_{M, \mathrm{Segal}}^{-\ast}(X^{\mathrm{an}}))^\wedge_{I_M} \stackrel{\simeq}{\too} K_{M, \mathrm{Borel}}^{-\ast}(X^{\mathrm{an}})\,,
\end{equation}
where $K_{M, \mathrm{Segal}}^\ast(-)$ stands for Segal's $M$-equivariant topological $K$-theory and $K_{M, \mathrm{Borel}}^\ast(-)$ for Borel's $M$-equivariant topological $K$-cohomology; $M$ is a(ny) maximal compact Lie subgroup of $G^{\mathrm{an}}$. Now, recall that Atiyah and Segal proved in \cite[Thm.~2.1 and Prop.~4.2]{AS} that, given any compact topological $M$-space $Y$ for which the $R(M)$-modules $K_{M, \mathrm{Segal}}^\ast(Y)$ are finitely generated, we have induced isomorphisms:
$$\theta_\ast\colon (K_{M, \mathrm{Segal}}^{\ast}(Y))^\wedge_I \stackrel{\simeq}{\too} K_{M, \mathrm{Borel}}^{\ast}(Y)\,.$$ 
Consequently, the above isomorphisms \eqref{eq:AS1} may be understood as an improvement of Atiyah-Segal's completion theorem in topological $K$-theory: the topological $M$-space $X^{\mathrm{an}}$ no longer needs to be compact and the $R(M)$-modules $K_{M, \mathrm{Segal}}^{-\ast}(X^{\mathrm{an}})$ no longer need to be finitely generated.
\end{remark}

\begin{remark}[$\bbQ$-coefficients]
Similarly to Remark \ref{rk:semi}, we have the following isomorphisms:
\begin{eqnarray}\label{eq:isos-topological}
\theta_n\colon (K_n^{\mathrm{top}}([X/G])_\bbQ)^\wedge_{I_G} \stackrel{\simeq}{\too} \mathrm{lim}_{j \geq 1} K_n^{\mathrm{top}}(X\times^G U_j)_\bbQ && n \in \bbZ\,.
\end{eqnarray}
Moreover, since Atiyah-Hirzebruch's spectral sequence \cite[\S2]{AH} degenerates rationally, we have an isomorphism between $K^{-n}_{\mathrm{top}}((X\times^G U_j)^{\mathrm{an}})_\bbQ$ and $\bigoplus_{i\,\mathrm{even}} H^i_{\mathrm{sing}}((X\times^G U_j)^{\mathrm{an}}; \bbQ)$, resp. $\bigoplus_{i\,\mathrm{odd}} H^i_{\mathrm{sing}}((X\times^G U_j)^{\mathrm{an}};\bbQ)$, when $n$ is even, resp. odd, where $H^\ast_{\mathrm{sing}}(-;\bbQ)$ stands for singular cohomology. Hence, using the fact that $K_n^{\mathrm{top}}(-) \simeq K^{-n}_{\mathrm{top}}((-)^{\mathrm{an}})$ and that Betti cohomology is defined as $H^\ast_{\mathrm{B}}(-):=H^\ast_{\mathrm{sing}}((-)^{\mathrm{an}};\bbQ)$, it follows from the definition of the $G$-equivariant Betti cohomology groups $H^i_{G, \mathrm{B}}(X)$ that \eqref{eq:isos-topological} yields the following isomorphisms:
\begin{eqnarray}\label{eq:iso-motivic-semi}
(K^{\mathrm{top}}_n([X/G])_\bbQ)^\wedge_{I_G} \simeq  \begin{cases} \prod_{i\,\mathrm{even}} H^i_{G, \mathrm{B}}(X) & \mathrm{if}\,\, n \,\,\mathrm{even} \\
 \prod_{i\,\mathrm{odd}} H^i_{G, \mathrm{B}}(X) & \mathrm{if}\,\, n \,\,\mathrm{odd} \,.
\end{cases}
\end{eqnarray}
To the best of the authors' knowledge, the isomorphisms \eqref{eq:iso-motivic-semi} are also new in the literature.
\end{remark}
\subsection{Periodic cyclic homology}\label{sub:HP}
Assume that $\mathrm{char}(k)=0$. Under this assumption, periodic cyclic homology gives rise to a localizing $\bbA^1$-homotopy invariant $HP(-)\colon \dgcat(k)_\infty \to \cC_{\bbZ/2}(k)_\infty$ with values in the $\infty$-category of $\bbZ/2$-graded complexes of $k$-vector spaces: condition (C1) follows from Keller's work \cite{Kel99}, condition (C2) follows from Goodwillie's work \cite{Goodwillie}, and condition (C3) follows from the fact that $HP(-)$ factors through the mixed complex $\infty$-functor (which preserve filtered colimits and satisfies condition (C1)); consult \cite[\S8.2.4 and \S8.2.7]{book}. Therefore, Theorem \ref{thm:completion} applied to $E=HP(-)$ and to the $\bbZ/2$-graded complex $o=k[v^{\pm 1}]$, with $|v|=2$, yields the following equivalence and isomorphisms:
\begin{eqnarray}\label{eq:equivalence-HP}
\theta\colon HP([X/T])^\wedge_I \too HP_T(X)  && \theta_\ast\colon (HP_\ast([X/T]))^\wedge_I \stackrel{\simeq}{\too} \mathrm{lim}_{j \geq 1} HP_\ast(X\times^T (\bbA^j\backslash\{0\})^r)\,.
\end{eqnarray}
To the best of the authors' knowledge, both $\theta$ as well as $\theta_\ast$ are new in the literature. 

Since $X\times^T (\bbA^j\backslash \{0\})^r$ is smooth, the Hochschild-Kostant-Rosenberg theorem (consult Feigin-Tsygan \cite{Feigin-Tsygan}) yields an isomorphism between $HP_n(X\times^T (\bbA^j\backslash\{0\})^r)$ and the direct sum $\bigoplus_{i\,\mathrm{even}} H^i_{\mathrm{dR}}(X\times^T (\bbA^j\backslash\{0\})^r)$, resp. $\bigoplus_{i\,\mathrm{odd}} H^i_{\mathrm{dR}}(X\times^T (\bbA^j\backslash\{0\})^r)$, when $n$ is even, resp. odd, where $H^\ast_{\mathrm{dR}}(-)$ stands for de Rham cohomology. Consequently, by definition of the $T$-equivariant de Rham cohomology groups $H^i_{T, \mathrm{dR}}(X)$, we obtain from the right-hand side of \eqref{eq:equivalence-HP} the following isomorphisms:
\begin{equation}\label{eq:cohomology}
(HP_n([X/T]))^\wedge_I \simeq  \begin{cases} \prod_{i\,\mathrm{even}} H^i_{T, \mathrm{dR}}(X) & \mathrm{if}\,\, n \,\,\mathrm{even} \\
 \prod_{i\,\mathrm{odd}} H^i_{T, \mathrm{dR}}(X) & \mathrm{if}\,\, n \,\,\mathrm{odd} \,.
\end{cases}
\end{equation}
To the best of the authors' knowledge, the isomorphisms \eqref{eq:cohomology} are also new in the literature.
\begin{remark}[Betti cohomology]
Assume that $k$ is equipped with an embedding $k \hookrightarrow \bbC$. In this case, we can consider the localizing $\bbA^1$-homotopy invariant $HP(-)\otimes_k \bbC$. As above, making use of Grothendieck's comparison isomorphism $H^\ast_{\mathrm{dR}}(-)\otimes_k \bbC \simeq H^\ast_{\mathrm{B}}(-)\otimes_\bbQ \bbC$ (consult \cite{Grothendieck}), we hence obtain the isomorphisms:
\begin{equation}\label{eq:cohomology1}
(HP_n([X/T])\otimes_k \bbC)^\wedge_I \simeq  \begin{cases} \prod_{i\,\mathrm{even}} H^i_{T, \mathrm{B}}(X)\otimes_\bbQ \bbC & \mathrm{if}\,\, n \,\,\mathrm{even} \\
 \prod_{i\,\mathrm{odd}} H^i_{T, \mathrm{B}}(X)\otimes_\bbQ \bbC & \mathrm{if}\,\, n \,\,\mathrm{odd} \,.
\end{cases}
\end{equation}
To the best of the authors' knowledge, the isomorphisms \eqref{eq:cohomology1} are also new in the literature. 
\end{remark}
\subsection*{Notations} Throughout the article, $k$ will denote a base field, $T$ a $k$-split algebraic torus, and $G$ a linearly reductive group $k$-scheme. Moreover, we will assume some basic familiarity with the languages of dg categories and $\infty$-categories; consult Keller's survey \cite{Keller} and Lurie's monographs \cite{Lurie1,Lurie2}, respectively. 
\section{Derived categories of quotient stacks}\label{sub:perfect}
Let $X$ a quasi-compact separated $k$-scheme equipped with a $T$-action. Throughout the article we will write $\mathrm{Mod}([X/T])$ for the Grothendieck category of $T$-equivariant $\cO_X$-modules, $\mathrm{Qcoh}([X/T])$ for the full subcategory of quasi-coherent $T$-equivariant $\cO_X$-modules, $\cD([X/T]):=\cD(\mathrm{Mod}([X/T]))$ for the derived category of $\mathrm{Mod}([X/T])$, $\cD_{\mathrm{Qcoh}}([X/T])\subset \cD([X/T])$ for the full triangulated subcategory of those complexes of $T$-equivariant $\cO_X$-modules whose cohomology belongs to $\mathrm{Qcoh}([X/T])$, and finally $\perf([X/T]) \subset \cD_{\mathrm{Qcoh}}([X/T])$ for the full triangulated subcategory of perfect complexes of $T$-equivariant $\cO_X$-modules. In the same vein, given a $T$-stable closed subscheme $Z \hookrightarrow X$, we will write $\cD([X/T])_Z$,  $\cD_{\mathrm{Qcoh}}([X/T])_Z$, and $\perf([X/T])_Z$, for the full triangulated subcategories of those complexes of $T$-equivariant $\cO_X$-modules that are (topologically) supported on $Z$.

Let $\cE$ be an exact category. As explained in \cite[\S4.4]{Keller}, the {\em derived dg category $\cD_\dg(\cE)$ of $\cE$} is defined as the dg quotient $\cC_\dg(\cE)/\cA{c}_\dg(\cE)$ of the dg category of complexes over $\cE$ by its full dg subcategory of acyclic complexes. Throughout the article we will write $\cD_\dg([X/T])$ for the dg category $\cD_\dg(\cE)$, with $\cE:=\mathrm{Mod}([X/T])$, and $\cD_{\mathrm{Qcoh}, \dg}([X/T])$ and $\perf_\dg([X/T])$ for its full dg subcategories. In the same vein, given a $T$-stable closed subscheme $Z\hookrightarrow X$, we will write $\cD_\dg([X/T])_Z$,  $\cD_{\mathrm{Qcoh}, \dg}([X/T])_Z$, and $\perf_\dg([X/T])_Z$, for the full dg subcategories of those complexes of $T$-equivariant $\cO_X$-modules that are supported on $Z$.

\begin{remark}[Generalization]\label{rk:generalization}
Given a $T$-equivariant sheaf of $\cO_X$-algebras $\mathcal{S}$, we can more generally consider the derived category $\cD([X/T];\cS)$ of $T$-equivariant $\mathcal{S}$-modules and also the associated dg category $\cD_{\mathrm{dg}}([X/T];\cS)$; similarly for all the other (triangulated and dg) categories. Note that in the particular case where $\cS=\cO_X$, these (triangulated and dg) categories reduce to the above ones.
\end{remark}


\begin{theorem}[Compact generation]\label{thm:compact}
Let $X$ be a quasi-compact separated $k$-scheme equipped with a $T$-action, and $Z \hookrightarrow X$ a $T$-stable closed subscheme. If $X$ is moreover geometrically normal and geometrically reduced, then the triangulated
  category $\cD_{\mathrm{Qcoh}}([X/T])_Z$ is compactly
  generated. Moreover, its full triangulated subcategory of compact
  objects identifies with $\perf([X/T])_Z$.
\end{theorem}
\begin{proof}
Thanks to Sumihiro's work \cite[Cor.~3.11]{Sumihiro1}\cite[Cor.~2]{Sumihiro2}, $X$ admits a $T$-stable Zariski {\em affine} open cover. Therefore, one can replicate the proof of \cite[Thm.~3.1.1]{BV}; consult also \cite{Neeman3}\cite[Tag 0AEC, Lem. 62.14.5]{Stacks}.
\end{proof}
\begin{proposition}\label{prop:exactseq}
Let $X$ be a quasi-compact separated $k$-scheme equipped with a $T$-action, $\mathrm{j}\colon V \hookrightarrow X$ a $T$-stable quasi-compact open subscheme, and $W\hookrightarrow X$ the closed complement of $V$. We assume moreover that $X$ is geometrically normal and geometrically reduced. Given a $T$-stable closed subscheme $Z \hookrightarrow X$ with quasi-compact open complement, we have the short exact sequence of dg categories:
$$ 0 \too \perf_\dg([X/T])_{Z\cap W} \too \perf_\dg([X/T])_Z \stackrel{\mathrm{j}^\ast}{\too} \perf_\dg([V/T])_{Z\cap V} \too 0\,. $$
\end{proposition}
\begin{proof}
Making use of Theorem \ref{thm:compact}, one can replicate the proof of \cite[Prop.~4.21]{Gysin}.
\end{proof}
\begin{theorem}[Excision]\label{thm:excision}
Let $\mathrm{f}\colon X' \to X$ be a $T$-equivariant flat map between quasi-compact separated $k$-schemes, and $Z \hookrightarrow X$ a $T$-stable closed subscheme such that $Z':=X'\times_X Z \to Z$ is an isomorphism. If $X$ and $X'$ are moreover geometrically normal and geometrically reduced, then we have a Morita equivalence:
$$ \mathrm{f}^\ast\colon \perf_\dg([X/T])_Z \too \perf_\dg([X'/T])_{Z'}\,.$$
\end{theorem}
\begin{proof}
Making use of Theorem \ref{thm:compact}, one can replicate the proof of \cite[Thm.~4.25]{Gysin}.
\end{proof}
\begin{remark}[Group scheme $G$]\label{rk:general}
Let $X$ be a separated $k$-scheme of finite type equipped with a $G$-action and $Z \hookrightarrow X$ a $G$-stable closed subscheme. If the quotient stack $[X/G]$ has the resolution property\footnote{Recall from Thomason \cite[Lem.~2.6]{Thomason-Adv} that if $X$ admits an ample family of $G$-equivariant line bundles, then $[X/G]$ has the resolution property. This holds, for example, when $X$ is normal and admits an ample family of (non-equivariant) line bundles.} (i.e., every coherent sheaf on $[X/G]$ is a quotient of a vector bundle), then, since $G$ is a linearly reductive group $k$-scheme, it follows from the work of Krishna-Ravi \cite[Lem. 2.7 and Prop.~3.3]{KR} that the triangulated
  category $\cD_{\mathrm{Qcoh}}([X/G])_Z$ is compactly
  generated. Moreover, its full triangulated subcategory of compact
  objects identifies with $\perf([X/G])_Z$. These two facts will play a key role in the proof of Theorem \ref{thm:homotopy}.
\end{remark}
\section{Noncommutative mixed motives}\label{sub:NCmotives}
Recall from \cite[\S8]{BGT}\cite[\S8.3]{book} the construction of the $\infty$-category of noncommutative mixed motives $\mathrm{NMot}(k)$. As explained in {\em loc. cit.}, $\mathrm{NMot}(k)$ is stable presentable. Moreover, it comes equipped with an $\infty$-functor $U\colon \dgcat(k)_\infty \to \mathrm{NMot}(k)$ which satisfies condition (C1), preserve filtered colimits, and is universal with respect to these properties, i.e., given any stable presentable $\infty$-category $\cD$, we have an induced equivalence
\begin{equation}\label{eq:induced111}
U^\ast\colon \mathrm{Fun}^L(\mathrm{NMot}(k), \cD) \too \mathrm{Fun}_{\mathrm{(C1)}, \mathrm{flt}}(\mathrm{dgcat}(k)_\infty, \cD)\,,
\end{equation}
where the left-hand side denotes the $\infty$-category of those $\infty$-functors which preserve colimits and the right-hand side the $\infty$-category of those $\infty$-functors which satisfy condition (C1) and preserve filtered colimits. 

As explained in \cite[\S8.5]{book}, the $\infty$-category $\mathrm{NMot}(k)$ admits an $\bbA^1$-homotopy variant $\mathrm{NMot}_{\bbA^1}(k)$. This $\infty$-category is also stable presentable and comes equipped with an $\infty$-functor $U_{\bbA^1}\colon \dgcat(k)_\infty \to \mathrm{NMot}_{\bbA^1}(k)$ which satisfies conditions (C1) and (C2), preserve filtered colimits, and is universal with respect to these properties, i.e., given any stable presentable $\infty$-category $\cD$, we have an induced equivalence:
\begin{equation}\label{eq:induced222}
U_{\bbA^1}^\ast\colon \mathrm{Fun}^L(\mathrm{NMot}_{\bbA^1}(k), \cD) \too \mathrm{Fun}_{\mathrm{(C1)}, \mathrm{(C2)}, \mathrm{flt}}(\mathrm{dgcat}(k)_\infty, \cD)\,.
\end{equation}
Finally, recall from \cite[\S8.3 and \S8.5]{book} that $U$ and $U_{\bbA^1}$ are moreover symmetric monoidal.
\begin{remark}\label{rk:key-remarks}
\begin{itemize}
\item[(i)] Thanks to the above equivalences \eqref{eq:induced111}-\eqref{eq:induced222}, the $\infty$-functor $U_{\bbA^1}$ factors through $U$ (via an $\infty$-functor which preserve filtered colimits).
\item[(ii)] Thanks to equivalence \eqref{eq:induced111}, every $\infty$-functor $E\colon \dgcat(k)_\infty \to \cD$ satisfying conditions (C1) and (C3) factors through $U$ (via an $\infty$-functor which does not necessarily preserve filtered colimits).
\end{itemize}
\end{remark}
\begin{proposition}\label{prop:Hom}
Let $X$ and $Y$ be two smooth quasi-compact separated $k$-schemes. If $X$ is moreover proper, then we have the following isomorphisms:
\begin{eqnarray}
\Hom_{\Ho(\mathrm{NMot}(k))}(U(X)[n],U(Y)) &\simeq & \Hom_{\Ho(\mathrm{NMot}_{\bbA^1}(k))}(U_{\bbA^1}(X)[n], U_{\bbA^1}(Y)) \label{eq:iso-equality} \\
&\simeq &  \begin{cases}  K_n(X\times Y) & n \geq 0 \\ 0 & n<0  \label{eq:iso-equality1}
 \end{cases}\,.
\end{eqnarray}
\end{proposition}
\begin{proof}
Recall from \cite[\S9.2]{BGT}\cite[\S8.6]{book} that given any two dg categories $\cA$ and $\cB$, with $\cA$ smooth and proper (consult \cite[\S1.7]{book}), we have natural equivalences of spectra
\begin{eqnarray}\label{eq:spectra}
\quad \mathrm{Map}_{\mathrm{NMot}(k)}(U(\cA),U(\cB))\simeq \bbK(\cA^\op \otimes \cB) &&
\mathrm{Map}_{\mathrm{NMot}_{\bbA^1}(k)}(U_{\bbA^1}(\cA),U_{\bbA^1}(\cB)) \simeq KH(\cA^\op \otimes \cB)\,,
\end{eqnarray}
where $\mathrm{Map}(-,-)$ stands for the mapping spectrum of the stable presentable $\infty$-category. Since $X$ is smooth and proper, the dg category $\perf(X)$ is smooth and proper. Hence, making use of the Morita equivalence $\perf_\dg(X)^\op \to  \perf_\dg(X), \cF \mapsto \uHom_X(\cF,\cO_X)$, and of the Morita equivalence (consult \cite[Lem.~4.26]{Gysin})
\begin{eqnarray*}
\perf_\dg(X) \otimes \perf_\dg(Y) \to \perf_\dg(X\times Y) && (\cF, \cG) \mapsto \cF\boxtimes \cG\,,
\end{eqnarray*}
we conclude from \eqref{eq:spectra} that the left hand-side, resp. right hand-side, of \eqref{eq:iso-equality} is naturally isomorphic to $\bbK_n(X\times Y)$, resp. $KH_n(X\times Y)$. Finally, since $X \times Y$ is smooth, the proof follows then from the well-known fact that the latter groups are both naturally isomorphic to the algebraic $K$-theory group $K_n(X\times Y)$ (which, by definition, is equal to $0$ when $n<0$).
\end{proof}
\begin{proposition}
Let $X$ and $Y$ be two smooth quasi-compact separated $k$-schemes equipped with a $T$-action. If $X$ is moreover proper and $T$ acts trivially on $X$, then we have the following isomorphisms:
\begin{eqnarray}
\quad \quad \Hom_{\Ho(\mathrm{NMot}(k))}(U([X/T])[n],U([Y/T])) &\simeq & \Hom_{\Ho(\mathrm{NMot}_{\bbA^1}(k))}(U_{\bbA^1}([X/T])[n], U_{\bbA^1}([Y/T])) \label{eq:iso-equality-T} \\
& \simeq  &\begin{cases}  \prod_{\chi \in \hat{T}} K_n([(X\times Y)/T]) & n \geq 0 \\ 0 & n<0 
 \end{cases}\,. \label{eq:iso-equality-2-T}
\end{eqnarray}
\end{proposition}
\begin{proof}
Recall first from \cite[Prop.~2.4]{Concentration} that, since $T$ acts trivially on $X$, we have the Morita equivalence:
\begin{eqnarray}\label{eq:Morita11}
\perf_\dg(X) \otimes \perf_\dg([\bullet/T]) \too \perf_\dg([X/T]) && (\cF,V) \mapsto \cF\boxtimes V\,.
\end{eqnarray}
We have the following isomorphisms $U([X/T]) \stackrel{\mathrm{(a)}}{\simeq} U(X) \otimes U([\bullet/T]) \stackrel{\mathrm{(b)}}{\simeq} U(X) \otimes \bigoplus_{\chi \in \hat{T}} U(k) \stackrel{\mathrm{(c)}}{\simeq} \bigoplus_{\chi \in \hat{T}} U(X)$, where (a) follows from \eqref{eq:Morita11} and from the fact that $U$ is symmetric monoidal, (b) follows from the canonical Morita equivalence $\coprod_{\chi \in \hat{T}} k \to \perf_\dg([\bullet/T])$ and from the fact that $U$ preserve filtered colimits, and (c) follows from the fact that the tensor product $- \otimes -$ on $\mathrm{NMot}(k)$ preserve colimits in each variable; similarly for $U_{\bbA^1}$. Therefore, making use of the Morita equivalence $\perf_\dg(X)^\op \to \perf_\dg(X), \cF \mapsto \uHom_X(\cF,\cO_X)$, and of the following Morita equivalence (obtained by replicating the proof of \cite[Lem.~4.26]{Gysin})
\begin{eqnarray*}
\perf_\dg(X) \otimes \perf_\dg([Y/T]) \too \perf_\dg([(X\times Y)/T]) && (\cF, \cG) \mapsto \cF\boxtimes \cG\,,
\end{eqnarray*} 
we conclude from \eqref{eq:spectra} that the left hand-side, resp. right hand-side, of \eqref{eq:iso-equality-T} is naturally isomorphic $\prod_{\chi \in \hat{T}} \bbK_n([(X\times Y)/T])$, resp. $\prod_{\chi \in \hat{T}} KH_n([(X\times Y)/T])$. Finally, since $X \times Y$ is smooth, the proof of \eqref{eq:iso-equality-T}, resp. \eqref{eq:iso-equality-2-T}, follows from Theorem \ref{thm:homotopy}, resp. from the fact that $\bbK_n([(X\times Y)/T])$ is naturally isomorphic to $K_n([(X\times Y)/T])$ (which, by definition, is equal to $0$ when $n<0$); consult \cite[\S9]{Schlichting}.
\end{proof}
\section{$K$-theory action}\label{sub:action}
We start with some generalities. Given a commutative monoid $o$ in
the symmetric monoidal $\infty$-category $\mathrm{NMot}(k)$, note that
the spectrum $\mathrm{Map}_{\mathrm{NMot(k)}}(U(k),o)$ becomes
naturally an $\bbE_\infty$-ring. Moreover, this $\bbE_\infty$-ring
acts on $o$ in the sense that we have the $\bbE_\infty$-ring map $\mathrm{Map}(U(k),o) \to \mathrm{Map}(o,o)$, $f\mapsto m\circ (f \otimes \id)$, where $m\colon o\otimes o \to o$ stands for the multiplication map.
\begin{proposition}\label{prop:action}
Let $X$ be a quasi-compact separated $k$-scheme equipped with a $T$-action. Given an $\infty$-functor $E\colon \dgcat(k)_\infty \to \cD$ which satisfies conditions $\mathrm{(C1)}$ and $\mathrm{(C3)}$, we have an induced action of the $\bbE_\infty$-ring $K([\bullet/T])$ on $E([X/T])$.
\end{proposition}
\begin{proof}
Recall first from Remark \ref{rk:key-remarks}(ii) that since $E$ satisfies conditions (C1) and (C3), there exists an $\infty$-functor $\overline{E}$ (which does not necessarily preserve filtered colimits) making the following diagram commute:
\begin{equation}\label{eq:factorization}
\xymatrix{
\dgcat(k)_\infty \ar[d]_-U \ar[rr]^-E && \cD \\
\mathrm{NMot}(k) \ar@/_1pc/[urr]_-{\overline{E}} && \,.
}
\end{equation}
Note that the tensor product $-\otimes_{\cO_X}-$ makes the dg category $\perf_\dg([X/T])$ into a commutative monoid in the symmetric monoidal $\infty$-category $\dgcat(k)_\infty$. Since $U$ is symmetric monoidal, we hence obtain a commutative monoid $U([X/T])$ in the symmetric monoidal $\infty$-category $\mathrm{NMot}(k)$. Under the natural equivalence of spectra
\begin{equation}\label{eq:Map}
\mathrm{Map}_{\mathrm{NMot}(k)}(U(k), U([X/T]))\simeq \bbK([X/T])\,,
\end{equation}
the $\bbE_\infty$-ring structure on the left-hand side of \eqref{eq:Map} induced by the commutative monoid structure of $U([X/T])$ corresponds to the classical $\bbE_\infty$-ring structure on $\bbK([X/T])$. Therefore, making use of the above generalities, we obtain an induced action of the $\bbE_\infty$-ring $\bbK([X/T])$ on $U([X/T])$. Now, precompose this action with the $\bbE_\infty$-ring map $K([\bullet/T]) \simeq \bbK([\bullet/T]) \stackrel{\mathrm{p}^\ast}{\to} \bbK([X/T])$, induced by the projection $\mathrm{p}\colon [X/T] \to [\bullet/T]$. This leads to an action of the $\bbE_\infty$-ring $K([\bullet/T])$ on $U([X/T])$. Then, make use of the above commutative diagram \eqref{eq:factorization} in order to obtain an induced action of the $\bbE_\infty$-ring $K([\bullet/T])$ on $\overline{E}(U([X/T]))\simeq E([X/T])$.
\end{proof}
\begin{remark}[Group scheme $G$]
Given a separated $k$-scheme of finite type $X$ equipped with a $G$-action, Proposition \ref{prop:action} holds {\em mutatis mutandis} with $T$ replaced by $G$.
\end{remark}
\section{Equivariant Gysin cofiber sequences}\label{sec:Gysin}
The next result, which is of independent interest, will play a key role in the proof of Theorem \ref{thm:completion}.
\begin{theorem}\label{thm:Gysin1}
Let $X$ be a smooth quasi-compact separated $k$-scheme equipped
with a $T$-action, $\mathrm{i}\colon Z \hookrightarrow X$ a $T$-stable smooth
closed subscheme, and $\mathrm{j}\colon U \hookrightarrow X$ the open
complement of $Z$. We assume moreover that $X$ is geometrically normal and geometrically reduced. Given a localizing $\bbA^1$-homotopy invariant $E\colon \dgcat(k)_\infty \to \cD$, we have an induced cofiber sequence of $K([\bullet/T])$-modules:
\begin{equation}\label{eq:cofiber}
E([Z/T]) \stackrel{\mathrm{i}_\ast}{\too} E([X/T]) \stackrel{\mathrm{j}^\ast}{\too} E([U/T])\,.
\end{equation}
\end{theorem}
\begin{corollary}\label{cor:Gysin1}
We have an induced cofiber sequence of $K([\bullet/T])$-modules:
\begin{equation}\label{eq:cofiber1}
E_T(Z) \stackrel{\mathrm{i}_\ast}{\too} E_T(X) \stackrel{\mathrm{j}^\ast}{\too} E_T(U)\,.
\end{equation}
\end{corollary}
\begin{proof}
Thanks to Theorem \ref{thm:Gysin1}, we have the following cofiber sequences of $K([\bullet/T])$-modules:
\begin{eqnarray*}
E([(Z\times (\bbA^j\backslash\{0\})^r)/T]) \stackrel{(\mathrm{i}\times \id)_\ast}{\too} E([(X\times (\bbA^j\backslash\{0\})^r)/T]) \stackrel{(\mathrm{j}\times \id)^\ast}{\too} E([(U\times (\bbA^j\backslash\{0\})^r)/T]) && j \geq 1\,.
\end{eqnarray*}
Consequently, since $\cD$ is a stable presentable $\infty$-category, by applying the $\infty$-functor $\mathrm{lim}_{j \geq 1}(-)$ to these cofiber sequences we obtain the above cofiber sequence \eqref{eq:cofiber1}.
\end{proof}
\subsection*{Proof of Theorem \ref{thm:Gysin1}}
Thanks to Proposition \ref{prop:exactseq}, we have the short exact sequence of dg categories:
$$ 0 \too \perf_\dg([X/T])_Z \too \perf_\dg([X/T]) \stackrel{\mathrm{j}^\ast}{\too} \perf_\dg([U/T]) \too 0\,.$$
Since $E$ satisfies condition (C1), we hence obtain the following cofiber sequence of $K([\bullet/T])$-modules
$$ E([X/T])_Z \too E([X/T]) \stackrel{\mathrm{j}^\ast}{\too} E([U/T])\,,$$
where $E([X/T])_Z:=E(\perf_\dg([X/T])_Z)$. Since the dg functor $\mathrm{i}_\ast\colon \perf_\dg([Z/T])\to \perf_\dg([X/T])$ factors through the inclusion $\perf_\dg([X/T])_Z\subset \perf_\dg([X/T])$, we have moreover a morphism of $K([\bullet/T])$-modules:
\begin{equation}\label{eq:induced-support}
\mathrm{i}_\ast\colon E([Z/T]) \too E([X/T])_Z\,.
\end{equation}
We claim that \eqref{eq:induced-support} is an equivalence; note that this claim would automatically conclude the proof of Theorem \ref{thm:Gysin1}. In what follows, we prove our claim. The proof is divided into two steps: affine case and general case.

\subsection*{Step I: affine case} In what follows, we assume that $X$ is {\em affine}. We start by recalling some key notions:
\begin{definition}[Grading]\label{def:gradings}
A dg category $\cA$ is called {\em $\bbN_0$-graded} if the (cochain) complexes of $k$-vector spaces $\cA(x,y)$ are equipped with a direct sum decomposition $\bigoplus_{n \geq 0} \cA(x,y)_n$ of (cochain) complexes of $k$-vector spaces, which is preserved by the composition law. Note that, by definition, the $\bbN_0$-grading of $\cA(x,y)$ is respected by the differential. The elements of $\cA(x,y)_n$ are called of {\em pure} degree $n$. Let $\cA_0$ be the dg category with the same objects as $\cA$ and $\cA_0(x,y):=\cA(x,y)_0$. Note that we have an ``inclusion'' dg functor $\cA_0\hookrightarrow \cA$.
\end{definition} 
\begin{remark}\label{rk:cohomology}
Let $\cA$ be a dg category whose (cochain) complexes of $k$-vector spaces $\cA(x,y)$ have zero differential and are supported in non-negative degrees. In this case, the dg category $\cA$ becomes $\bbN_0$-graded: an element of $\cA(x,y)$ is of pure degree $n$ if it is of cohomological degree $n$.
\end{remark}
\begin{lemma}[Consult Lemma 6.6 of \cite{Gysin}]\label{lem:gradings}
Given an $\infty$-functor $E\colon \dgcat(k)_\infty \to \cD$ which satisfies condition $\mathrm{(C2)}$, we have an induced equivalence $E(\cA_0) \to E(\cA)$.
\end{lemma}
\begin{definition}
A dg category $\cA$ is called {\em formal} if it is isomorphic to the (cohomology) dg category $H^\ast(\cA)$ in the homotopy category $\Ho(\dgcat(k)_\infty)$.
\end{definition}
The following result, which is of independent interest, extends \cite[Thm.~6.8]{Gysin} to the equivariant setting.
\begin{theorem}[Equivariant formality]\label{thm:formality} 
Let $X$ be a smooth affine $k$-scheme equipped with a $T$-action, $Z \hookrightarrow X$ a $T$-stable smooth affine closed subscheme, and $\cI \subset \cO_X$ the defining ideal of $Z$ in $X$. We assume that $X$ is geometrically normal and geometrically reduced. Under these assumptions, the following holds:
\begin{itemize}
\item[(i)] The family of $T$-equivariant sheaves $\{(\cO_X/\cI)\otimes_k \chi\}_{\chi \in \hat{T}} \in \perf([X/T])_Z$ is a family of compact generators of the triangulated category $\cD_{\mathrm{Qcoh}}([X/T])_Z$. In the same vein, $\{(\cO_X/\cI)\otimes_k \chi\}_{\chi \in \hat{T}} \in \perf([Z/T])$ is a family of compact generators of the triangulated category $\cD_{\mathrm{Qcoh}}([Z/T])$.
\item[(ii)] The full dg subcategory $\cA$ of $\perf_\dg([X/T])_Z$ spanned by the generators $\{(\cO_X/\cI)\otimes_k \chi\}_{\chi \in \hat{T}}$ is formal. Moreover, given any two objects $\cF$ and $\cG$ of $\cA$, we have $H^n(\perf_\dg([X/T])_Z(\cF,\cG))=0$ for every $n<0$ and $H^0(\perf_\dg([X/T])_Z(\cF,\cG))=H^0(\perf_\dg([Z/T])(\cF,\cG))$.
\item[(iii)] The full dg subcategory $\cB$ of $\perf_\dg([Z/T])$ spanned by the generators $\{(\cO_X/\cI)\otimes_k \chi\}_{\chi \in \hat{T}}$ is formal. Moreover, given any two objects $\cF$ and $\cG$ of $\cB$, we have $H^n(\perf_\dg([Z/T])(\cF,\cG))=0$ for every $n\neq 0$.
\end{itemize}
\end{theorem}
\begin{proof}
(i) We prove only the first claim; the second claim is similar. Given an object $\cF \in \cD_{\mathrm{Qcoh}}([X/T])_Z$, we need to show that if the following equality
\begin{equation}\label{eq:zero}
\Hom_{\cD_{\mathrm{Qcoh}}([X/T])_Z}((\cO_X/\cI)\otimes_k \chi,\cF[n])=0
\end{equation}
holds for every $\chi \in \hat{T}$ and $n \in \bbZ$, then $\cF\simeq 0$.
Since $X$ is affine, the preceding equality \eqref{eq:zero} may be re-written as $\big(\underline{\Hom}_X(\cO_X/\cI, \cF)\otimes_k \chi^{-1}\big)^T=0$. This implies that if the latter equality holds for every $\chi \in \hat{T}$, then $\underline{\Hom}_X(\cO_X/\cI, \cF)=0$. Now, by replicating the proof of \cite[Thm.~6.8(i)]{Gysin}, we conclude that $\cF\simeq 0$.

(ii) It is now convenient to switch to the setting of $k$-algebras
equipped with a $T$-action, i.e., to the setting of $\hat{T}$-graded $k$-algebras. Note that most concepts and results concerning $k$-algebras extend trivially to $\hat{T}$-graded $k$-algebras. Let $X=\mathrm{Spec}(R)$, $Z=\mathrm{Spec}(S)$, and $\phi\colon R \twoheadrightarrow S$ the surjective $\hat{T}$-graded $k$-algebra homomorphism, with kernel $I$, corresponding to the $T$-equivariant closed immersion $Z \hookrightarrow X$. In the same vein, let $\cD(T;R)$ be the derived category of $T$-equivariant $R$-modules and $\cD(T;R)_I$ the full triangulated subcategory of those complexes of $T$-equivariant $R$-modules whose cohomology is locally annihilated by a
power of $I$. Thanks to \cite[Thm.~1.2]{HallNeemanRydh}, the category $\cD(T;R)$ is equivalent to
$\cD_{\mathrm{Qcoh}}([X/T])$. Moreover, this equivalence restricts to an equivalence between $\cD(T;R)_I$ and $\cD_{\mathrm{Qcoh}}([X/T])_Z$.

Now, consider the $T$-equivariant completion $\widehat{R}$ of $R$ at the ideal $I$ in the sense of Magid \cite{MR920520}. Consider also the $T$-equivariant completion $\widehat{\mathrm{Sym}}$ of the $\bbN_0$-graded symmetric algebra $\mathrm{Sym}:=\mathrm{Sym}_S(I/I^2)$ (with $\mathrm{Sym}_n:=\mathrm{Sym}^n_S(I/I^2)$) at the ideal $\mathrm{Sym}_{\geq 1}$. By replicating the proof of \cite[Prop.~6.12]{Gysin}, using $T$-equivariant completion instead of classical completion, we conclude that $I/I^2$ is a finitely generated projective $S=R/I$-module and also that there exists a $T$-equivariant isomorphism $\tau\colon \widehat{\mathrm{Sym}} \stackrel{\simeq}{\to} \widehat{R}$ such that $\widehat{\phi}\circ \tau$ agrees with the projection onto $\mathrm{Sym}_0=S$. This implies, in particular, that $\tau$ yields an induced equivalence of categories $\cD(T;\widehat{R})_{\widehat{I}} \simeq \cD(T;\widehat{\mathrm{Sym}})_{\widehat{\mathrm{Sym}}_{\geq 1}}$. By combining the above considerations, we hence obtain the following equivalences of triangulated categories
\begin{equation}\label{eq:equivalences}
  \cD_{\mathrm{Qcoh}}([X/T])_Z \simeq \cD(T;R)_I \stackrel{\mathrm{(a)}}{\simeq} \cD(T;\widehat{R})_{\widehat{I}} \simeq \cD(T;\widehat{\mathrm{Sym}})_{\widehat{\mathrm{Sym}}_{\geq 1}} \stackrel{\mathrm{(b)}}{\simeq} \cD(T;\mathrm{Sym})_{\mathrm{Sym}_{\geq 1}}\,,
\end{equation}
where (a) and (b) follow from Theorem \ref{thm:excision}. This yields an isomorphism in $\Ho(\dgcat(k)_\infty)$ between the dg categories $\cD_{\mathrm{Qcoh}, \dg}([X/T])_Z$ and $\cD_{\dg}(T;\mathrm{Sym})_{\mathrm{Sym}_{\geq 1}}$. Under this isomorphism, the dg category $\cA$ corresponds to the full dg subcategory $\cA'$ of $\cD_{\dg}(T;\mathrm{Sym})_{\mathrm{Sym}_{\geq 1}}$ spanned by the objects $\{S\otimes_k \chi\}_{\chi \in \hat{T}}$. Let $\mathbb{K} \to S$ be the Koszul resolution of $S$ as a $\mathrm{Sym}$-module and $A:=\mathbf{R}\mathrm{End}_{\mathrm{Sym}}(\mathbb{K})$ the associated $\hat{T}$-graded dg $k$-algebra. Clearly, $\cA'$ is isomorphic in $\Ho(\dgcat(k)_\infty)$ to the full dg subcategory of $\cD_{\dg}(T;\widehat{\mathrm{Sym}})$ spanned by the objects $\{\mathbb{K}\otimes_k \chi\}_{\chi \in \hat{T}}$. Moreover, this latter dg category is isomorphic in $\Ho(\dgcat(k)_\infty)$ to the full dg subcategory $\cA''$ of $\cD_\dg(T;A)$ spanned by the objects $\{A\otimes_k \chi\}_{\chi \in \hat{T}}$. Now, by replicating the proof of \cite[Prop.~6.13]{Gysin}, using $\hat{T}$-graded $k$-algebras instead of $k$-algebras, we conclude that $A$ is $T$-equivariantly formal and that $H^n(A)=0$ for every $n<0$. This implies that $\cA''$ is furthermore isomorphic in $\Ho(\dgcat(k)_\infty)$ to the full dg subcategory $\cA'''$ of $\cD_\dg(T;H^\ast(A))$ spanned by the objects $\{H^\ast(A)\otimes_k \chi\}_{\chi \in \hat{T}}$. The proof follows now from the fact that the (cochain) complexes of $k$-vector spaces of the dg category $\cA'''$ have zero differential and are supported in non-negative degrees. Finally, the claim that $H^0(\perf_\dg([X/T])_Z(\cF,\cG))=H^0(\perf_\dg([Z/T])(\cF,\cG))$ for any two objects $\cF$ and $\cG$ of $\cA$ is now clear from the above arguments.

(iii) The proof is similar to the proof of item (ii).
\end{proof}
We now have all the ingredients necessary to conclude the proof of Step I. Thanks to Theorem \ref{thm:formality}(i), the inclusions of dg categories $\cA\subset \perf_\dg([X/T])_Z$ and $\cB \subset \perf_\dg([Z/T])$ are Morita equivalences. Moreover, thanks to Theorem \ref{thm:formality}(ii)-(iii), the dg categories $\cA$ and $\cB$ are formal, $H^n(\cA)=0, n<0$, and $H^0(\cA)=\cB$. Therefore, in the homotopy category $\Ho(\dgcat(k)_\infty)$, the dg functor $\mathrm{i}_\ast\colon \perf_\dg([Z/T]) \to \perf_\dg([X/T])_Z$ corresponds to the inclusion $H^0(\cA) \hookrightarrow H^\ast(\cA)$. Making use of Remark \ref{rk:cohomology} and Lemma \ref{lem:gradings}, we hence conclude that the above morphism of $K([\bullet/T])$-modules \eqref{eq:induced-support} is an equivalence.

\subsection*{Step II: general case}
\begin{proposition}\label{prop:square}
Let $X$ be a smooth quasi-compact separated $k$-scheme equipped with a $T$-action, $X = V_1 \cup V_2$ a $T$-stable Zariski open cover of $X$, and $\mathrm{i}\colon Z \hookrightarrow X$ a $T$-stable closed subscheme. We assume that $X$ is geometrically normal and geometrically reduced. Given an $\infty$-functor $E \colon \dgcat(k)_\infty \to \cD$ which satisfies condition (C1), we have the (co)cartesian square (the morphisms are induced by the open inclusions)
\begin{equation}\label{eq:square111}
\xymatrix{
E([X/T])_Z  \ar[r] \ar[d] & E([V_1/T])_{Z_1} \ar[d] \\
E([V_2/T])_{Z_2} \ar[r] & E([V_{12}/T])_{Z_{12}}\,,
}
\end{equation}
where $V_{12}:=V_1 \cap V_2$, $Z_1:= Z\cap V_1$, $Z_2:= Z\cap V_2$, and $Z_{12}:= Z\cap V_{12}$.
\end{proposition}
\begin{proof}
Let us write $W$ for the (reduced) closed complement $(X\backslash V_1)_{\mathrm{red}}$ of $V_1$ and $W_2 := W\cap V_2$. Under these notations, we have the following commutative diagram:
\begin{equation}\label{eq:diagram-dg}
\xymatrix{
\perf_\dg([X/T])_{Z\cap W} \ar[d] \ar[r] & \perf_\dg([X/T])_Z \ar[d] \ar[r] & \perf_\dg([V_1/T])_{Z_1} \ar[d] \\
\perf_\dg([V_2/T])_{Z_2 \cap W_2} \ar[r] & \perf_\dg([V_2/T])_{Z_2} \ar[r] & \perf_\dg([V_{12}/T])_{Z_{12}}\,.
}
\end{equation}
Thanks to Proposition \ref{prop:exactseq}, both rows in \eqref{eq:diagram-dg} are short exact sequences of dg categories. Moreover, since the inclusion $V_2 \hookrightarrow X$ restricts to an isomorphism $Z_2 \cap W_2 \to Z\cap W$, Theorem \ref{thm:excision} implies that the left-hand side vertical dg functor in \eqref{eq:diagram-dg} is a Morita equivalence. Consequently, since $E$ satisfies condition (C1), we obtain the following commutative diagram
$$
\xymatrix{
E([X/T])_{Z\cap W} \ar[d] \ar[r] & E([X/T])_Z \ar[d] \ar[r] & E([V_1/T])_{Z_1} \ar[d] \\
E([V_2/T])_{Z_2 \cap W_2} \ar[r] & E([V_2/T])_{Z_2} \ar[r] & E([V_{12}/T])_{Z_{12}}\,,
}
$$
where each row is a cofiber sequence and the left-hand side vertical morphism is an equivalence. Since $\cD$ is a stable presentable $\infty$-category, we hence conclude that \eqref{eq:square111} is (co)cartesian.
\end{proof}
\begin{corollary}\label{cor:square}
Let $X$ be a smooth quasi-compact separated $k$-scheme equipped with a $T$-action, $X = V_1 \cup V_2$ a $T$-stable Zariski open cover of $X$, and $\mathrm{i}\colon Z \hookrightarrow X$ a $T$-stable closed subscheme. We assume that $X$ is geometrically normal and geometrically reduced. Given an $\infty$-functor $E \colon \dgcat(k)_\infty \to \cD$ which satisfies condition (C1), if the following morphisms are equivalences
\begin{eqnarray*}
\mathrm{i}_\ast^1 \colon E([Z_1/T]) \too  E([V_1/T])_{Z_1} & \mathrm{i}_\ast^2 \colon E([Z_2/T]) \too  E([V_2/T])_{Z_2} & \mathrm{i}_\ast^{12} \colon E([Z_{12}/T]) \too E([V_{12}/T])_{Z_{12}}\,,
\end{eqnarray*}
then the morphism $\mathrm{i}_\ast\colon E([Z/T])\to E([X/T])_Z$ is also an equivalence.
\end{corollary}
\begin{proof}
Consider the the following commutative diagram:
\begin{equation}\label{eq:bigsquares}
\xymatrix@C=2.5em@R=2em{
E([X/T])_Z \ar[rrr] \ar[ddd]& & & E([V_1/T])_{Z_1} \ar[ddd] \\
& E([Z/T]) \ar[ul]_-{\mathrm{i}_\ast}  \ar[r] \ar[d]& E([Z_1/T]) \ar[ur]^-{\mathrm{i}^1_\ast} \ar[d] & \\
& E([Z_2/T]) \ar[dl]_-{\mathrm{i}^2_\ast} \ar[r] & E([Z_{12}/T]) \ar[dr]^-{\mathrm{i}^{12}_\ast} & \\
E([V_2/T])_{Z_2} \ar[rrr] & & & E([V_{12}/T])_{Z_{12}}\,.
}
\end{equation}
Note that, thanks to Proposition \ref{prop:square}, both the inner and outer squares in \eqref{eq:bigsquares} are (co)cartesian. Therefore, since by hypothesis $\mathrm{i}_\ast^1$, $\mathrm{i}_\ast^2$ and $\mathrm{i}_\ast^{12}$ are equivalences, we conclude that $\mathrm{i}_\ast$ is also an equivalence.
\end{proof}
We now have all the ingredients necessary to conclude the proof of Step II (and hence of Theorem \ref{thm:Gysin1}). Thanks to Sumihiro's work \cite[Cor.~3.11]{Sumihiro1}\cite[Cor.~2]{Sumihiro2}, $X$ admits a $T$-stable Zariski {\em affine} open cover $\{V_i\}_{i \in I}$. Moreover, the quasi-compactness of $X$ implies that this cover admits a finite subcover $\{V_i\}_{i=1}^n$. Therefore, an inductive argument using Step I and Corollary \ref{cor:square} allows us to conclude that the above morphism of $K([\bullet/T])$-modules \eqref{eq:induced-support} is an equivalence.

\section{Equivariant vector bundles}\label{sec:bundle}
The following result, which is of independent interest, will play a key role in the proof of Theorem \ref{thm:completion}.
\begin{theorem}\label{thm:fibration}
Let $Z$ be a quasi-compact separated normal $k$-scheme equipped with a $T$-action and $\mathrm{q}\colon W \to Z$ a $T$-equivariant vector bundle. Given a localizing $\bbA^1$-homotopy invariant $E\colon \dgcat(k)_\infty \to \cD$, we have an induced equivalence of $K([\bullet/T])$-modules:
\begin{equation}\label{eq:equivalence}
\mathrm{q}^\ast \colon E([Z/T]) \too E([W/T])\,.
\end{equation} 
\end{theorem}
\begin{corollary}\label{cor:fibration}
We have an induced equivalence of $K([\bullet/T])$-modules:
\begin{equation}\label{eq:equivalence11}
\mathrm{q}^\ast \colon E_T(Z) \too E_T(W)\,.
\end{equation} 
\end{corollary}
\begin{proof}
Thanks to Theorem \ref{thm:fibration}, we have the following equivalences of $K([\bullet/T])$-modules:
\begin{eqnarray*}
(q\times \id)^\ast \colon E([(Z\times(\bbA^j\backslash\{0\})^r)/T]) \too E([(W\times(\bbA^j\backslash\{0\})^r)/T]) && j \geq 1\,.
\end{eqnarray*}
Therefore, by applying the $\infty$-functor $\mathrm{lim}_{j \geq 1}(-)$ to these equivalences, we obtain the equivalence \eqref{eq:equivalence11}.
\end{proof}
\subsection*{Proof of Theorem \ref{thm:fibration}}
By definition, the $T$-equivariant vector bundle $W$ is given by $\mathrm{Spec}_Z(\mathrm{Sym}(\cF))$ for some $T$-equivariant locally free sheaf $\cF$ on $Z$. Therefore, the dg category $\perf_\dg([W/T])$ identifies with the dg category of perfect complexes of $T$-equivariant $\mathrm{Sym}(\cF)$-modules $\perf_\dg([Z/T];\mathrm{Sym}(\cF))$; consult Remark \ref{rk:generalization}. Since the $T$-equivariant sheaf of $\cO_Z$-algebras $\mathrm{Sym}(\cF)$ is $\mathbb{N}_0$-graded and $\mathrm{Sym}(\cF)_0=\cO_Z$, we have an inclusion map $\iota\colon \cO_Z \hookrightarrow  \mathrm{Sym}(\cF)$. Under these notations, the above morphism of $K([\bullet/T])$-modules \eqref{eq:equivalence} may be re-written as $\iota^\ast\colon E(\perf_\dg([Z/T];\cO_Z)) \to E(\perf_\dg([Z/T];\mathrm{Sym}(\cF)))$. Consider the associated adjunction:
\begin{equation}\label{eq:adjunction1}
\xymatrix{
\cD_{\mathrm{Qcoh}}([Z/T]; \mathrm{Sym}(\cF)) \ar@<1ex>[d]^{\iota_\ast}\\
\cD_{\mathrm{Qcoh}}([Z/T]; \cO_Z) \ar@<1ex>[u]^{\iota^\ast} \,.
}
\end{equation}
The functor $\iota^\ast$ preserve perfect complexes while the functor $\iota_\ast$ preserve arbitrary direct sums and is moreover conservative. Therefore, given a set of perfect (=compact) generators $\{\cG_i\}_{i \in I}$ of the triangulated category $\cD_{\mathrm{Qcoh}}([Z/T])$ (recall from Theorem \ref{thm:compact} that such a set of generators exists), we conclude that $\{\iota^\ast(\cG_i)\}_{i \in I}$ is a set of perfect (and hence compact) generators of $\cD_{\mathrm{Qcoh}}([Z/T];\mathrm{Sym}(\cF))$. This implies that the full dg subcategory $\cA$ of $\perf_\dg([Z/T]; \mathrm{Sym}(\cF))$ spanned by the generators $\{\iota^\ast(\cG_i)\}_{i \in I}$ is Morita equivalent to $\perf_\dg([Z/T]; \mathrm{Sym}(\cF))$. In the same vein, the full dg subcategory $\cB$ of $\perf_\dg([Z/T])$ spanned by $\{\cG_i\}_{i \in I}$ is Morita equivalent to $\perf_\dg([Z/T]; \cO_Z)$. Moreover, given $i, i' \in I$, have the following natural identifications
\begin{eqnarray}
{\bf R}\Hom(\iota^\ast(\cG_i), \iota^\ast(\cG_{i'})) &\simeq & {\bf R}\Hom(\cG_i, \iota_\ast \iota^\ast(\cG_{i'})) \nonumber \\
& \simeq & {\bf R}\Hom(\cG_i, \cG_{i'} \otimes \mathrm{Sym}(\cF)) \nonumber \\
& \simeq & \oplus_{n \geq 0} {\bf R}\Hom(\cG_i, \cG_{i'} \otimes \mathrm{Sym}(\cF)_n)\,, \label{eq:iso-compact}
\end{eqnarray}
where \eqref{eq:iso-compact} follows from the compactness of $\cG_i$. This implies that the dg category $\cA$ inherits from $\mathrm{Sym}(\cF)$ a $\bbN_0$-grading, in the sense of Definition \ref{def:gradings}, and that $\cA_0=\cB$. Therefore, in the homotopy category $\Ho(\dgcat(k)_\infty)$, the dg functor $\iota^\ast\colon \perf_\dg([Z/T];\cO_Z) \to \perf_\dg([Z/T];\mathrm{Sym}(\cF))$ corresponds to the inclusion $\cA_0 \hookrightarrow \cA$. Making use of the above Lemma \ref{lem:gradings}, we hence conclude that $\iota^\ast=\mathrm{q}^\ast$ is an equivalence.
\begin{remark}[Generalization]
Note that if one ignores the $K([\bullet/T])$-module structure, then Theorem \ref{thm:fibration} holds more generally for every $\infty$-functor $E\colon \dgcat(k)_\infty \to \cD$ which satisfies (solely) condition (C2).
\end{remark}
\begin{remark}[Group scheme $G$]
Given a separated $k$-scheme of finite type $Z$ equipped with a $G$-action, Theorem \ref{thm:fibration} holds {\em mutatis mutandis} with $T$ replaced by $G$: simply replace Theorem \ref{thm:compact} by Remark \ref{rk:general}.
\end{remark}
\section{Proof of Theorem \ref{thm:completion}}
The proof of Theorem \ref{thm:completion} is divided into two steps: trivial action and general case.
\subsection*{Step I: trivial action}
In what follows, we assume that $T$ acts trivially on $X$.

\subsection*{Proof of item (i).} The augmentation ideal $I$ of the representation ring $R(T)\simeq \bbZ[\hat{T}]\simeq \bbZ[t_1^{\pm}, \ldots, t_r^\pm]$ is generated by the elements $1-t_1, \ldots, 1-t_r$. Therefore, thanks to \cite[Cor.~4.2.12]{Lurie}, in order to prove that the $K([\bullet/T])$-module $E_T(X)$ is $I$-complete, it suffices to show that $E_T(X)$ is $\langle (1-t_i)\rangle$-complete for every $1 \leq i \leq r$. Note that by definition of $E_T(X)$, it is enough to show that each one of the $K([\bullet/T])$-modules $E([(X\times (\bbA^j\backslash\{0\})^r)/T]), j \geq 1$, is $\langle (1-t_i)\rangle$-complete for every $1 \leq i \leq r$. Concretely, it is enough to show that the following limit\footnote{Following Lurie \cite[Rk.~4.1.11]{Lurie}, the morphism $-\cdot (1-t_i)$ in \eqref{eq:limit} is defined as $-\cdot \overline{(1-t_i)}$, where $\overline{(1-t_i)}$ is a(ny) point of the $\bbE_\infty$-ring $K([\bullet/T])$ which belongs to the path-connected component $(1-t_i) \in \pi_0 K([\bullet/T])=R(T)$. The above limit \eqref{eq:limit} is independent of this choice.} vanishes:
\begin{equation}\label{eq:limit}
\mathrm{lim}\big(\cdots \stackrel{-\cdot (1-t_i)}{\too} E([(X\times (\bbA^j\backslash\{0\})^r)/T]) \stackrel{-\cdot (1-t_i)}{\too} E([(X\times (\bbA^j\backslash\{0\})^r)/T]) \big)\,.
\end{equation}
Consider the commutative diagram:
\begin{equation}\label{eq:square}
\xymatrix{
X\times (\bbA^j\backslash \{0\})^r \ar[r] \ar[d]_-{\mathrm{p}_j} & (\bbA^j\backslash \{0\})^r \ar[d] \\
X \ar[r] & \bullet\,.
}
\end{equation}
Since $T$ acts trivially on $X$, we have $X\times^T (\bbA^j\backslash \{0\})^r = X\times (\bbP^{j-1})^r$. Moreover, all the maps in \eqref{eq:square} are $T$-equivariant. Therefore, by applying the functor $K_0(-)$ to \eqref{eq:square}, we obtain the commutative diagram:
\begin{eqnarray}\label{eq:squares}
\xymatrix{
K_0(X\times (\bbP^{j-1})^r) & K_0((\bbP^{j-1})^r) \ar[l] \\
K_0([X/T]) \ar[u]^-{\mathrm{p}_j^\ast} & R(T) \ar[l] \ar[u]\,.
}
\end{eqnarray}
By combining the commutative square \eqref{eq:squares} with the following short exact sequence
\begin{equation*}
0 \too \langle (1-t_1)^j, \ldots, (1-t_r)^j\rangle \too R(T) \too K_0((\bbP^{j-1})^r) \too 0\,,
\end{equation*}
we hence conclude from the $K([\bullet/T])$-module structure of $E([(X\times (\bbA^j\backslash\{0\})^r)/T])$ that the following limit 
$$
\mathrm{lim}\big(\cdots \stackrel{-\cdot (1-t_i)^j}{\too} E([(X\times (\bbA^j\backslash\{0\})^r)/T]) \stackrel{-\cdot (1-t_i)^j}{\too} E([(X\times (\bbA^j\backslash\{0\})^r)/T]) \big)
$$
vanishes. In other words, we conclude that the $K([\bullet/T])$-module $E([(X\times (\bbA^j\backslash\{0\})^r)/T])$ is $\langle (1-t_i)^j\rangle$-complete. The proof follows now from the general fact that a $K([\bullet/T])$-module is $\langle (1-t_i)^j\rangle$-complete if and only if is $\langle (1-t_i)\rangle$-complete.

\subsection*{Proof of item (ii) - case of a circle.} In what follows, we assume that $T$ is a circle, i.e., $T=\bbG_m$. Recall that in this particular case the representation ring $R(T)$ is isomorphic to $\bbZ[t^{\pm 1}]$. In order to simplify the exposition, we will make essential use of the following notation:
\begin{notation}
Given an additive category $\cC$ (e.g., the homotopy category $\Ho(\cD)$ of a stable presentable $\infty$-category), let us write $-\otimes_\bbZ-$ for the canonical action of the category of free $\bbZ$-modules~on~$\cC$.
\end{notation}
Note that in order to prove that the induced morphism of $K([\bullet/T])$-modules $\theta\colon E([X/T])^\wedge_I \to E_T(X)$ is an equivalence, it suffices to show that the following homomorphisms of abelian groups are invertible:
\begin{eqnarray*}
\pi_{o,n}(\theta) \colon \pi_{o,n}(E([X/T])^\wedge_I) \too \pi_{o,n}(E_T(X)) && o \in \cD \quad n \in \bbZ\,.
\end{eqnarray*}
As explained in \cite[Prop.~4.2.7]{Lurie}, the $K([\bullet/T])$-module $E([X/T])^\wedge_I$ admits the following description: 
\begin{eqnarray*}
E([X/T])^\wedge_I = \mathrm{lim}_{j\geq 1} \frac{E([X/T])}{(1-t)^{j}} & \mathrm{where} & \frac{E([X/T])}{(1-t)^{j}} := \mathrm{cofiber}\left(E([X/T]) \stackrel{-\cdot (1-t)^j}{\too} E([X/T]) \right)\,.
\end{eqnarray*}
Consequently, we have the Milnor short exact sequence:
\begin{equation}\label{eq:Milnor}
0 \too \mathrm{lim}^1_{j \geq 1} \pi_{o,n+1}\frac{E([X/T])}{(1-t)^j} \too  \pi_{o,n}(E([X/T])^\wedge_I)
 \too  \mathrm{lim}_{j \geq 1} \pi_{o,n}\frac{E([X/T])}{(1-t)^j} \too 0\,.
\end{equation}
Thanks to Proposition \ref{prop:key-completion}, we have moreover the following isomorphisms of abelian groups: 
\begin{equation}\label{eq:iso-abelian}
\pi_{o,n}\frac{E([X/T])}{(1-t)^j} \simeq \pi_{o,n}E(X) \otimes_\bbZ R(T)/\langle (1-t)^j\rangle 
 \simeq \pi_{o,n}E(X) \otimes_\bbZ \bbZ[t]/\langle (1-t)^j\rangle \,.
\end{equation}
Under \eqref{eq:iso-abelian}, the tower of abelian groups $\{\pi_{o,n}\frac{E([X/T])}{(1-t)^j}\}_{j\geq 1}$ corresponds to $\pi_{o,n}E(X) \otimes_\bbZ \{\bbZ[t]/\langle (1-t)^j\rangle\}_{j \geq 1}$. Since the latter tower satisfies the Mittag-Leffler condition, we hence conclude that $\mathrm{lim}^1_{j \geq 1}\pi_{o,n+1}\frac{E([X/T])}{(1-t)^j}=0$. Consequently, the above Milnor short exact sequence \eqref{eq:Milnor} yields an induced isomorphism:
\begin{equation}\label{eq:iso-1}
\pi_{o,n}(E([X/T])^\wedge_I) \stackrel{\simeq}{\too} \mathrm{lim}_{j\geq 1}\big(\pi_{o,n} E(X)\otimes_\bbZ R(T)/\langle (1-t)^j \rangle\big)\,.
\end{equation}  
Since $T$ acts trivially on $X$, we have $X\times^T \bbA^j\backslash\{0\} = X\times \bbP^{j-1}$. Therefore, the $K([\bullet/T])$-module $E_T(X)$ admits the following description $E_T(X)= \mathrm{lim}_{j \geq 1} E(X\times \bbP^{j-1})$. This implies, in particular, that we have the following Milnor short exact sequence:
\begin{equation}\label{eq:Milnor1}
    0 \too \mathrm{lim}^1_{j \geq 1} \pi_{o,n+1} E(X\times \bbP^{j-1}) \too  \pi_{o,n} E_T(X)
               \too   \mathrm{lim}_{j \geq 1} \pi_{o,n} E(X\times \bbP^{j-1}) \too  0\,.
               \end{equation}
Thanks to Proposition \ref{prop:key-completion}, we have moreover the following isomorphisms of abelian groups: 
\begin{equation}\label{eq:iso-abelian11}
\pi_{o,n} E(X\times \bbP^{j-1}) \simeq  \pi_{o,n}E(X) \otimes_\bbZ K_0(\bbP^{j-1})\simeq \pi_{o,n}E(X)\otimes_\bbZ \bbZ^{\oplus j}\,.
\end{equation}
Under \eqref{eq:iso-abelian11}, the tower of abelian groups $\{\pi_{o,n} E(X\times \bbP^{j-1})\}_{j \geq 1}$ corresponds to $\pi_{o,n} E(X) \otimes_\bbZ \{\bbZ^{\oplus j}\}_{j \geq 1}$. Since the latter tower satisfies the Mittag-Leffler condition, we hence conclude that $\mathrm{lim}^1_{j \geq 1} \pi_{o,n+1} E(X\times \bbP^{j-1})=0$. Consequently, the above Milnor short exact sequence \eqref{eq:Milnor1} yields an induced isomorphism:
\begin{equation}\label{eq:iso-2}
\pi_{o,n} E_T(X) \stackrel{\simeq}{\too} \mathrm{lim}_{j \geq 1} \big( \pi_{o,n} E(X) \otimes_\bbZ K_0(\bbP^{j-1}) \big)\,.
\end{equation}
Finally, under \eqref{eq:iso-1} and \eqref{eq:iso-2}, the above homomorphism $\pi_{o,n}(\theta)$ corresponds to the isomorphism
$$ \mathrm{lim}_{j \geq 1} \big(\pi_{o,n}E(X) \otimes_\bbZ R(T)/\langle (1-t)^j \rangle\big) \stackrel{\simeq}{\too} \mathrm{lim}_{j \geq 1}  \big(\pi_{o,n} E(X) \otimes_\bbZ  K_0(\bbP^{j-1})\big)$$
induced by the following short exact sequences:
\begin{eqnarray*}
0 \too \langle (1-t)^j\rangle \too R(T) \too K_0(\bbP^{j-1}) \too 0&& j \geq 1\,.
\end{eqnarray*}
This implies, in particular, that the homomorphism $\pi_{o,n}(\theta)$ is invertible.
\begin{proposition}\label{prop:key-completion}
We have natural isomorphisms
\begin{eqnarray*}
\frac{E([X/T])}{(1-t)^j} \simeq E(X) \otimes_\bbZ R(T)/\langle (1-t)^j\rangle  &&
E(X\times \bbP^{j-1}) \simeq E(X) \otimes_\bbZ K_0(\bbP^{j-1})
\end{eqnarray*}
in the homotopy category $\Ho(\cD)$.
\end{proposition}
\begin{proof}
Thanks to Remark \ref{rk:key-remarks}(ii), there exists an $\infty$-functor $\overline{E}$ (which does not necessarily preserve filtered colimits) making the following diagram commute:
\begin{equation}\label{eq:diagram-last}
\xymatrix{
\dgcat(k)_\infty \ar[d]_-U \ar[rr]^-E && \cD \\
\mathrm{NMot}(k) \ar@/_1pc/[urr]_-{\overline{E}} && \,.
}
\end{equation}
Since the induced functor $\Ho(\overline{E})$ is triangulated, it preserve (finite) direct sums. This implies that it is compatible with the canonical action $-\otimes_\bbZ-$ of the category of {\em finitely generated} free $\bbZ$-modules. Consequently, using the fact that the abelian groups $R(T)/\langle (1-t)^j\rangle$ and $K_0(\bbP^{j-1})$ are finitely generated, it suffices then to show that we have natural isomorphisms 
\begin{eqnarray}\label{iso-key1}
\frac{U([X/T])}{(1-t)^j} \simeq U(X) \otimes_\bbZ R(T)/\langle (1-t)^j\rangle &&
U(X\times \bbP^{j-1}) \simeq  U(X) \otimes_\bbZ K_0(\bbP^{j-1})
\end{eqnarray}
in the homotopy category $\Ho(\mathrm{NMot}(k))$. We start by constructing the left-hand side of \eqref{iso-key1}. Since $T$ acts trivially on $X$, we have the following Morita equivalence (consult \cite[Prop.~2.4]{Concentration}):
\begin{eqnarray}\label{eq:Morita-last}
\perf_\dg(X) \otimes \perf_\dg([\bullet/T]) \too \perf_\dg([X/T]) && (\cF,V) \mapsto \cF \boxtimes V\,.
\end{eqnarray}
Making use of it, we hence obtain the following natural isomorphisms
\begin{equation}\label{eq:composed}
U([X/T])\stackrel{\mathrm{(a)}}{\simeq} U(X) \otimes U([\bullet/T]) \stackrel{\mathrm{(b)}}{\simeq} U(X) \otimes_\bbZ R(T)\,,
\end{equation}
where (a) follows from the Morita equivalence \eqref{eq:Morita-last} and from the fact that $U$ is symmetric monoidal, and (b) from Lemma \ref{lem:aux-last1}. Since the abelian group $R(T)/\langle (1-t)^j\rangle\simeq \bbZ^{\oplus j}$ is free, the short exact sequence
\begin{equation}\label{eq:split}
0 \too R(T) \stackrel{-\cdot (1-t)^j}{\too} R(T) \too R(T)/\langle (1-t)^j\rangle \too 0
\end{equation}
is split. By functoriality, this yields the following split cofiber sequence 
$$ U(X) \otimes_\bbZ R(T) \stackrel{\id\otimes_\bbZ (-\cdot (1-t)^j)}{\too} U(X) \otimes_\bbZ R(T) \too U(X) \otimes_\bbZ R(T)/\langle (1-t)^j\rangle$$
in the homotopy category $\Ho(\mathrm{NMot}(k))$. Consequently, there exists a (unique) dashed isomorphism in $\Ho(\mathrm{NMot}(k))$ making the following diagram commute:
$$
\xymatrix{
U([X/T]) \ar[rrr]^-{-\cdot (1-t)^{j}} &&& U([X/T]) \ar[r] & \frac{U([X/T])}{(1-t)^{j}} \\
U(X) \otimes_\bbZ R(T) \ar[u]^-{\eqref{eq:composed}}_-\simeq \ar[rrr]_-{\id\otimes_\bbZ (-\cdot (1-t)^{j})} &&& U(X) \otimes_\bbZ R(T) \ar[r] \ar[u]^-{\eqref{eq:composed}}_-\simeq & U(X)\otimes_\bbZ R(T)/\langle (1-t)^{j}\rangle \ar@{-->}[u]_-{\simeq}\,. 
}
$$
Finally, in what concerns the isomorphism on the right-hand side of \eqref{iso-key1}, it is given by the composition $
U(X\times \bbP^{j-1}) \stackrel{\text{(a)}}{\simeq} U(X) \otimes U(\bbP^{j-1}) \stackrel{\text{(b)}}{\simeq} U(X) \otimes_\bbZ K_0(\bbP^{j-1})$,
where (a) follows from the Morita equivalence 
\begin{eqnarray*}
\perf_\dg(X) \otimes \perf_\dg(\bbP^{j-1}) \too \perf_\dg(X\times \bbP^{j-1}) && (\cF, \cG)\mapsto \cF\boxtimes \cG
\end{eqnarray*}
and from the fact that $U$ is symmetric monoidal, and (b) from Lemma \ref{lem:aux-last1}.
\end{proof}
\begin{remark}
Note that it follows from the proof of Proposition \ref{prop:key-completion} that the following cofiber sequences
\begin{eqnarray}\label{eq:split-sequences}
E([X/T]) \stackrel{-\cdot (1-t^j)}{\too} E([X/T]) \too \frac{E([X/T])}{(1-t)^j} && j \geq 1
\end{eqnarray}
are {\em split} in the homotopy category $\Ho(\cD)$.
\end{remark}
\begin{lemma}\label{lem:aux-last1}
We have natural isomorphisms
\begin{eqnarray*}
U([\bullet/T]) \simeq U(k) \otimes_\bbZ R(T) &&
U(\bbP^{j-1})\simeq  U(k) \otimes_\bbZ K_0(\bbP^{j-1})
\end{eqnarray*}
in the homotopy category $\Ho(\mathrm{NMot}(k))$.
\end{lemma}
\begin{proof}
The computation $\Hom_{\Ho(\mathrm{NMot}(k))}(U(k), U([\bullet/T]))\simeq R(T)$ (consult \S\ref{sub:NCmotives}) yields an induced morphism $U(k)\otimes_\bbZ R(T) \to U([\bullet/T])$. Similarly, the computation $\Hom_{\Ho(\mathrm{NMot}(k))}(U(k), U(\bbP^{j-1})) \simeq K_0(\bbP^{j-1})$ yields an induced morphism $U(k) \otimes_\bbZ K_0(\bbP^{j-1}) \to  U(\bbP^{j-1})$. 
By applying the functor $\Hom_{\Ho(\mathrm{NMot}(k))}(U(k), -)$ to these morphisms, we obtain the corresponding computations. Consequently, thanks to the Yoneda lemma, it is enough to show that $U([\bullet/T])$ and $U(\bbP^{j-1})$ are isomorphic to a (possibly infinite) direct sum of copies of $U(k)$. In what concerns $U([\bullet/T])$, this follows from the canonical Morita equivalence $\coprod_{\chi \in \hat{T}}k \to \perf_\dg([\bullet/T])$ and from the fact that $U$ preserve filtered colimits. In what concerns $U(\bbP^{j-1})$, this follows from Beilinson's celebrated full exceptional collection $\perf(\bbP^{j-1}) = \langle \cO, \cO(1), \ldots, \cO(j)\rangle$ and from the fact that $U$ sends full exceptional collections to direct sums of copies of $U(k)$; consult \cite{Beilinson}\cite[\S2.4.2 and \S8.4.5]{book}.
\end{proof}
\subsection*{Proof of item (ii) - case of a torus.} In what follows, we assume that $T=\bbG_m^r$. Recall that in this case the representation ring $R(T)$ is isomorphic to $\bbZ[t_1^{\pm1}, \ldots, t_r^{\pm1}]$. As explained in \cite[Prop.~4.2.7]{Lurie}, the $K([\bullet/T])$-module $E([X/T])^\wedge_I$ admits the following description 
$$E([X/T])^\wedge_I = \mathrm{lim}_{j_1, \ldots, j_r \geq 1} \frac{E([X/T])}{(1-t_1)^{j_1} + \cdots + (1- t_r)^{j_r}}\,,$$ 
where $\frac{E([X/T])}{(1-t_1)^{j_1} + \cdots + (1- t_i)^{j_i}}$, with $1 \leq i \leq r$, is defined recursively as follows:
$$\mathrm{cofiber}\left(\frac{E([X/T])}{(1-t_1)^{j_1} + \cdots + (1- t_{i-1})^{j_{i-1}}} \stackrel{-\cdot (1-t_i)^{j_i}}{\too} \frac{E([X/T])}{(1-t_1)^{j_1} + \cdots + (1- t_{i-1})^{j_{i-1}}}\right)\,.$$ 
In the same vein, since $X\times^T(\bbA^j\backslash\{0\})^r= X\times(\bbP^{j-1})^r$, the $K([\bullet/T])$-module $E_T(X)$ admits the description:
$$E_T(X):=\mathrm{lim}_{j\geq 1} E(X\times (\bbP^{j-1})^r)= \mathrm{lim}_{j_1, \ldots, j_r\geq 1} E(X\times \bbP^{j_1-1} \times \cdots \times \bbP^{j_r-1})\,.$$
Now, note that the morphism of $K([\bullet/T])$-modules $\theta\colon E([X/T])^\wedge_I \to E_T(X)$ admits the factorization
$$ E([X/T])^\wedge_I \stackrel{\theta^1}{\too} E_T(X)^1 \stackrel{\theta^2}{\too} E_T(X)^2 \stackrel{\theta^3}{\too} \cdots \stackrel{\theta^r}{\too} E_T(X)^r= E_T(X)\,,$$
where
\begin{eqnarray*}
E_T(X)^i:= \mathrm{lim}_{j_1, \ldots, j_r \geq 1} \frac{E(X\times \bbP^{j_1-1} \times \cdots \times \bbP^{j_i-1})}{(1-t_{i+1})^{j_{i+1}} + \cdots + (1-t_r)^{j_r}} && 1 \leq i \leq r
\end{eqnarray*}
and $\theta^i$ is induced by the projection $X\times \bbP^{j_1-1} \times \cdots \times \bbP^{j_i-1} \to X \times \bbP^{j_1-1} \times \cdots \times \bbP^{j_{i-1}-1}$. In order to prove that $\theta$ is an equivalence, it suffices then to show that $\theta^i$, with $1 \leq i \leq r$, is an equivalence. For each choice of integers $j_1, \ldots, j_{i-1}, j_{i+1}, \ldots, j_r$, it follows from item (ii) (case of a circle) that the induced morphism
$$ \mathrm{lim}_{j_i \geq 1} \frac{E(X\times \bbP^{j_1-1} \times \cdots \times \bbP^{j_{i-1}-1})}{(1-t_{i})^{j_{i}} + \cdots + (1-t_r)^{j_r}} \too  \mathrm{lim}_{j_i \geq 1} \frac{E(X\times \bbP^{j_1-1} \times \cdots \times \bbP^{j_i-1})}{(1-t_{i+1})^{j_{i+1}} + \cdots + (1-t_r)^{j_r}}
$$
is an equivalence; note that since $\cD$ is a stable presentable $\infty$-category, the $\infty$-functor $\mathrm{lim}_{j_i\geq 1}(-)$ preserve (co)fiber sequences. Consequently, we conclude that $\theta^i$ is also an equivalence.
\subsection*{Proof of item (iii) - case of a circle.}
As explained in the proof of item (ii), we have natural isomorphisms: 
\begin{eqnarray*}
\pi_{o,n}E([X/T])\simeq \pi_{o,n}E(X) \otimes_\bbZ R(T) && \pi_{o,n}E(X\times \bbP^{j-1})\simeq \pi_{o,n}E(X) \otimes_\bbZ K_0(\bbP^{j-1})\,.
\end{eqnarray*}
These yield the following description of the classical completion and classical limit of abelian groups:
\begin{eqnarray}
(\pi_{o,n} E([X/T]))^\wedge_I & \simeq & \mathrm{lim}_{j\geq 1}\big(\pi_{o,n} E(X)\otimes_\bbZ R(T)/\langle (1-t)^j \rangle\big) \label{eq:description1} \\
\mathrm{lim}_{j\geq 1} \pi_{o,n} E(X\times \bbP^{j-1}) & \simeq & \mathrm{lim}_{j \geq 1} \big( \pi_{o,n} E(X) \otimes_\bbZ K_0(\bbP^{j-1}) \big)\,. \label{eq:description2}
\end{eqnarray}
The proof follows now from the fact that, under \eqref{eq:description1}, resp. \eqref{eq:description2}, the left-hand side, resp. right-hand side, of \eqref{eq:completion} corresponds to the above isomorphism \eqref{eq:iso-1}, resp. \eqref{eq:iso-2}.
\subsection*{Proof of item (iii) - case of a torus.} Similarly to Proposition \ref{prop:key-completion}, we have natural isomorphisms
\begin{eqnarray*}
\frac{E([X/T])}{(1-t_1)^{j_1}+ \cdots + (1-t_r)^{j_r}} & \simeq & E(X) \otimes_\bbZ \big(R(T)/\langle (1-t_1)^{j_1} + \cdots + (1-t_r)^{j_r} \rangle\big) \\
& \simeq & E(X) \otimes_\bbZ \big(\bbZ[t_1]/\langle (1-t_1)^{j_1}\rangle \otimes \cdots \otimes \bbZ[t_r]/\langle (1-t_r)^{j_r}\rangle  \big)\,,
\end{eqnarray*}
as well as natural isomorphisms
$$
E(X\times \bbP^{j_1-1} \times \cdots \times \bbP^{j_r-1}) \simeq E(X) \otimes_\bbZ K_0(\bbP^{j_1-1} \times \cdots \times \bbP^{j_r-1}) \simeq E(X)\otimes_Z \bbZ^{\oplus j_1} \otimes \cdots \otimes \bbZ^{\oplus j_r}\,,
$$
in the homotopy category $\Ho(\cD)$. This implies that the following multi-towers of abelian groups 
\begin{eqnarray*}
\{\pi_{o,n}\frac{E([X/T])}{(1-t_1)^{j_1} + \cdots + (1-t_r)^{j_r}}\}_{j_1, \ldots, j_r \geq 1} && \{\pi_{o,n} E(X\times \bbP^{j_1-1} \times \cdots \times \bbP^{j_r-1})\}_{j_1, \ldots, j_r \geq 1}
\end{eqnarray*}
corresponds to the following multi-towers:
$$
\pi_{o,n}E(X) \otimes_\bbZ \{\bbZ[t_1]/\langle (1-t_1)^{j_1}\rangle \otimes \cdots \otimes \bbZ[t_1]/\langle (1-t_r)^{j_r}\rangle\}_{j_1, \ldots, j_r \geq 1} \quad \pi_{o,n}E(X) \otimes_\bbZ \{\bbZ^{\oplus j_1}\otimes \cdots \otimes \bbZ^{\oplus j_r}\}_{j_1, \ldots, j_r \geq 1}\,.
$$
Since the latter multi-towers of abelian groups satisfy the Mittag-Leffler condition (in each one of the $r$ possible directions), we hence obtain induced isomorphisms:
\begin{eqnarray}
\pi_{o,n}(E([X/T])^\wedge_I) & \stackrel{\simeq}{\too} & \mathrm{lim}_{j_1, \ldots, j_r\geq 1} \big(\pi_{o,n}E(X) \otimes_\bbZ R(T)/\langle (1-t_1)^{j_1} + \cdots + (1-t_r)^{j_r} \rangle \big) \label{eq:induced-last1} \\
\pi_{o,n} E_T(X) &  \stackrel{\simeq}{\too} & \mathrm{lim}_{j_1, \ldots, j_r \geq 1} \big( \pi_{o,n}E(X) \otimes_\bbZ K_0(\bbP^{j_1-1} \times \cdots \times \bbP^{j_r-1}) \big)\,.  \label{eq:induced-last2} 
\end{eqnarray}
Similarly to item (ii) (case of a circle), we have moreover natural isomorphisms:
\begin{eqnarray*}
& \pi_{o,n}E([X/T])\simeq \pi_{o,n}E(X) \otimes_\bbZ R(T) & \pi_{o,n}E(X\times (\bbP^{j-1})^r)\simeq \pi_{o,n}E(X) \otimes_\bbZ K_0((\bbP^{j-1})^r)\,.
\end{eqnarray*}
They yield the following description of the classical completion of abelian groups
\begin{eqnarray}
(\pi_{o,n} E([X/T]))^\wedge_I & \simeq & \mathrm{lim}_{j_1, \ldots, j_r\geq 1}\big(\pi_{o,n} E(X)\otimes_\bbZ R(T)/\langle (1-t_1)^{j_1} + \cdots + (1-t_r)^{j_r} \rangle\big) \label{eq:description11}
\end{eqnarray}
as well as the following description of the classical limit of abelian groups
\begin{eqnarray}
\mathrm{lim}_{j\geq 1} \pi_{o,n} E(X\times (\bbP^{j-1})^r) & \simeq & \mathrm{lim}_{j \geq 1} \big( \pi_{o,n} E(X) \otimes_\bbZ K_0((\bbP^{j-1})^r) \big) \label{eq:description22} \\
& \simeq & \mathrm{lim}_{j_1, \ldots, j_r \geq 1} \big( \pi_{o,n} E(X) \otimes_\bbZ K_0(\bbP^{j_1-1} \times \cdots \times \bbP^{j_r-1}) \big) \label{eq:description222} \,.
\end{eqnarray}
The proof follows now from the fact that, under \eqref{eq:description11}, resp. \eqref{eq:description22}-\eqref{eq:description222}, the left-hand side, resp. right-hand side, of \eqref{eq:completion} corresponds to the above isomorphism \eqref{eq:induced-last1}, resp. \eqref{eq:induced-last2}.
\subsection*{Step II: general case}
In what follows, we assume that $X$ is $T$-filtrable.

\subsection*{Proof of item (i).} By combining the filtration \eqref{eq:filtration} with Corollary \ref{cor:Gysin1}, we obtain the following cofiber sequences of $K([\bullet/T])$-modules
\begin{eqnarray}\label{eq:sequences}
E_T(W_i) \too E_T(X\backslash X_{i-1})\too E_T(X\backslash X_i) && 0 \leq i \leq m-1\,,
\end{eqnarray}
where $W_i:=X_i\backslash X_{i-1}$ is a $T$-stable smooth closed subscheme of $X\backslash X_{i-1}$ and $X\backslash X_i$ is the open complement of $W_i$; note that $X\backslash X_{m-1}=W_m$. Thanks to Corollary \ref{cor:fibration}, we have moreover equivalences of $K([\bullet/T])$-modules $\mathrm{q}_i^\ast\colon E_T(Z_i)\to E_T(W_i), 0 \leq i \leq m$. Therefore, since the $\infty$-subcategory of $I$-complete $K([\bullet/T])$-modules is stable under cofiber sequences (consult \cite[\S4.2]{Lurie}), an inductive argument using the cofiber sequences \eqref{eq:sequences}, the equivalences $\mathrm{q}_i^\ast, 0 \leq i \leq m$, and the fact that the $K([\bullet/T])$-modules $E_T(Z_i), 0 \leq i \leq m$, are $I$-complete (as proved in item (i) of Step I), allows us to conclude that the $K([\bullet/T])$-module $E_T(X)$ is $I$-complete.

\subsection*{Proof of item (ii).} By combining the filtration \eqref{eq:filtration} with Theorem \ref{thm:Gysin1}, we obtain the following cofiber sequences of $K([\bullet/T])$-modules:
\begin{eqnarray*}
E([W_i/T])\too E([(X\backslash X_{i-1})/T])\too E([(X\backslash X_i)/T]) && 0 \leq i \leq m-1\,,
\end{eqnarray*}
where $W_i:=X_i\backslash X_{i-1}$ is a $T$-stable smooth closed subscheme of $X\backslash X_{i-1}$ and $X\backslash X_i$ is the open complement of $W_i$; note that $X\backslash X_{m-1}=W_m$. Therefore, since the derived completion functor $(-)^\wedge_I$ preserve cofiber sequences (consult \cite[\S4.2]{Lurie}), we obtain the following cofiber sequences of $K([\bullet/T])$-modules:
\begin{eqnarray*}
E([W_i/T])^\wedge_I\too E([(X\backslash X_{i-1})/T])^\wedge_I\too E([(X\backslash X_i)/T])^\wedge_I && 0 \leq i \leq m-1\,.
\end{eqnarray*}
Moreover, we have the following morphism of cofiber sequences of $K([\bullet/T])$-modules:
\begin{equation}\label{eq:diagrams}
\xymatrix{
E([W_i/T])^\wedge_I \ar[d]_-{\theta_{W_i}} \ar[r] & E([(X\backslash X_{i-1})/T])^\wedge_I \ar[r] \ar[d]^-{\theta_{X\backslash X_{i-1}}} & E([(X\backslash X_i)/T])^\wedge_I \ar[d]^-{\theta_{X\backslash X_{i}}} \\
E_T(W_i) \ar[r] & E_T(X\backslash X_{i-1})\ar[r] & E_T(X\backslash X_i)\,,
}
\end{equation}
where the commutativity of the left-hand side square, resp. right-hand side square, follows from the compatibility of $\theta$ with push-forwards, resp. with pull-backs. Furthermore, Theorem \ref{thm:fibration} and Corollary \ref{cor:fibration} yield the following commutative squares
\begin{equation}\label{eq:squares1}
\xymatrix{
E([Z_i/T])^\wedge_I \ar[d]_-{\theta_{Z_i}} \ar[rr]^-{\mathrm{q}_i^\ast} && E([W_i/T])^\wedge_I \ar[d]^-{\theta_{W_i}}\\
E_T(Z_i) \ar[rr]_-{\mathrm{q}_i^\ast}&& E_T(W_i) \,,
} 
\end{equation}
where both horizontal morphisms are equivalences. Therefore, an inductive argument using the commutative diagrams \eqref{eq:diagrams}-\eqref{eq:squares1} and the fact that the morphisms $\theta_{Z_i}, 0 \leq i \leq m$, are equivalences (as proved in item (ii) of Step I), allows us to conclude that the morphism $\theta_X\colon E([X/T])^\wedge_I \to E_T(X)$ is also an equivalence.

\subsection*{Proof of item (iii) - case of a circle} 
Recall from Step I that we have the Milnor short exact sequence:
\begin{equation}\label{eq:Milnor-new}
0 \too \mathrm{lim}^1_{j \geq 1} \pi_{o,n+1} \frac{E([X/T])}{(1-t)^j} \too \pi_{o,n}(E([X/T])^\wedge_I) \too \mathrm{lim}_{j\geq 1} \pi_{o,n} \frac{E([X/T])}{(1-t)^j} \too 0\,.
\end{equation}
Since $X$ is $T$-filtrable, by combining the filtration \eqref{eq:filtration} with Theorems \ref{thm:Gysin1} and \ref{thm:fibration}, we obtain the following cofiber sequences of $K([\bullet/T])$-modules:
\begin{eqnarray}\label{eq:cofiber-key}
E([Z_i/T]) \too E([(X\backslash X_{i-1})/T]) \too E([(X\backslash X_i)/T]) && 0 \leq i \leq m-1\,;
\end{eqnarray}
note that since $X\backslash X_{m-1}=W_m$, the $K([\bullet/T])$-module $E([(X\backslash X_{m-1})/T])$ is equivalent to $E([Z_m/T])$. As proved in Proposition \ref{prop:key1}, the cofiber sequences \eqref{eq:cofiber-key} are split in the homotopy category $\Ho(\cD)$. Consequently, the induced cofiber sequences
\begin{eqnarray}\label{eq:quocient}
\frac{E([Z_i/T])}{(1-t)^j} \too \frac{E([(X\backslash X_{i-1})/T])}{(1-t)^j} \too \frac{E([(X\backslash X_i)/T])}{(1-t)^j} && j \geq 1
\end{eqnarray} 
are also split in $\Ho(\cD)$. This leads to the following towers of short exact sequences of abelian groups
$$ 0 \too \{\pi_{o,n} \frac{E([Z_i/T])}{(1-t)^j}\}_{j\geq 1} \too \{\pi_{o,n}  \frac{E([(X\backslash X_{i-1})/T])}{(1-t)^j}\}_{j\geq 1} \too \{ \pi_{o,n} \frac{E([(X\backslash X_i)/T])}{(1-t)^j} \}_{j\geq 1} \too 0$$
and, as a byproduct, to the following exact sequences of abelian groups:
\begin{equation*}\label{eq:lim-last}
\mathrm{lim}^1_{j\geq 1}\pi_{o,n+1} \frac{E([Z_i/T])}{(1-t)^j} \too \mathrm{lim}^1_{j \geq 1} \pi_{o,n+1}    \frac{E([(X\backslash X_{i-1})/T])}{(1-t)^j} \too \mathrm{lim}^1_{j\geq 1} \pi_{o,n+1} \frac{E([(X\backslash X_i)/T])}{(1-t)^j} \too 0\,.
\end{equation*}
Therefore, an inductive argument using the latter exact sequences of abelian groups and the fact that $\mathrm{lim}^1_{j\geq 1}\pi_{o,n+1} \frac{E([Z_i/T])}{(1-t)^j}=0$ for every $0 \leq i \leq m$ (as proved in item (ii) of Step I), allows us to conclude that $\mathrm{lim}^1_{j\geq 1} \pi_{o,n+1} \frac{E([X/T])}{(1-t)^j}=0$. Thanks to Milnor's short exact sequence \eqref{eq:Milnor-new}, we hence obtain an isomorphism: 
\begin{equation}\label{eq:induced1}
\pi_{o,n} (E([X/T])^\wedge_I) \stackrel{\simeq}{\too} \mathrm{lim}_{j \geq 1} \pi_{o,n} \frac{E([X/T])}{(1-t)^j}\,.
\end{equation}
Now, consider the following cofiber sequences 
\begin{eqnarray*}
E([X/T])\stackrel{-\cdot (1-t)^j}{\too} E([X/T]) \too \frac{E([X/T])}{(1-t)^j} && j \geq 1
\end{eqnarray*}
and the associated universal coefficient short exact sequences
\begin{equation}\label{eq:seq-universal}
0 \too \pi_{o,n}E([X/T])/(1-t)^j \too \pi_{o,n} \frac{E([X/T])}{(1-t)^j} \too \mathrm{Tor}_{(1-t)^j}\pi_{o,n-1}E([X/T]) \too 0\,,
\end{equation}
where $\mathrm{Tor}_{(1-t)^j}\pi_{o,n-1}E([X/T])$ stands for the $(1-t)^j$-torsion $R(T)$-submodule of $\pi_{o,n-1}E([X/T])$. Note that since the above cofiber sequences \eqref{eq:cofiber-key} are split in the homotopy category $\Ho(\cD)$, they yield the following short exact sequences of $R(T)$-modules:
$$
0 \too \pi_{o, n-1} E([Z_i/T]) \too \pi_{o, n-1} E([(X\backslash X_{i-1})/T]) \too \pi_{o, n-1} E([(X\backslash X_i)/T]) \too 0\,. 
$$
As a byproduct, we obtain the following exact sequences:
$$ 0 \too \mathrm{Tor}_{(1-t)^j}\pi_{o, n-1} E([Z_i/T]) \too \mathrm{Tor}_{(1-t)^j}\pi_{o, n-1} E([(X\backslash X_{i-1})/T]) \too \mathrm{Tor}_{(1-t)^j} \pi_{o, n-1} E([(X\backslash X_i)/T])\,.$$
Therefore, an inductive argument using these latter exact sequences of $R(T)$-modules and the fact that $\mathrm{Tor}_{(1-t)^j} \pi_{o,n-1}E([Z_i/T])=0$ for every $1 \leq i \leq m$ (this follows automatically from the above split cofiber sequences \eqref{eq:split-sequences} (with $X$ replaced by $Z_i$)), allows us to conclude that $\mathrm{Tor}_{(1-t)^j} \pi_{o,n-1}E([X/T])=0$. Thanks to the above universal coefficients short exact sequence \eqref{eq:seq-universal}, we hence obtain an induced isomorphism:
\begin{equation}\label{eq:induced2}
(\pi_{o,n}E([X/T]))^\wedge_I = \mathrm{lim}_{j \geq 1}\big(\pi_{o,n}E([X/T])/(1-t)^j\big) \stackrel{\simeq}{\too} \mathrm{lim}_{j \geq 1} \pi_{o,n} \frac{E([X/T])}{(1-t)^j}\,.
\end{equation}
The left-hand side of \eqref{eq:completion} is now defined as the composition of \eqref{eq:induced1} with the inverse of \eqref{eq:induced2}. 

We now construct the right-hand side of \eqref{eq:completion}. Recall that we have the Milnor short exact sequence:
\begin{equation}\label{eq:Milnor-new2}
0 \too \mathrm{lim}_{j\geq 1}^1 \pi_{o,n+1}E(X\times^T \bbA^j\backslash\{0\}) \too \pi_{o,n}E_T(X) \too \mathrm{lim}_{j\geq 1} \pi_{o,n}E(X\times^T \bbA^j\backslash\{0\}) \too 0\,.
\end{equation}
Since $X$ is $T$-filtrable, by combining the filtration \eqref{eq:filtration} with Theorems \ref{thm:Gysin1} and \ref{thm:fibration}, we obtain the following cofiber sequences of $K([\bullet/T])$-modules
\begin{eqnarray}\label{eq:seq-key2}
& \quad E(Z_i\times^T \bbA^j\backslash\{0\}) \too E((X\backslash X_{i-1})\times^T \bbA^j\backslash\{0\}) \too E((X\backslash X_i)\times^T \bbA^j\backslash\{0\}) & \quad 0 \leq i \leq m-1\,;
\end{eqnarray} 
note that $X\backslash X_{m-1}=W_m$. As proved in Proposition \ref{prop:key1}, these cofiber sequences are split in the homotopy category $\Ho(\cD)$. Hence, they yield the towers of short exact sequences of abelian groups
$$ 0 \to \{\pi_{o,n}E(Z_i\times^T \bbA^j\backslash\{0\})\}_{j\geq 1} \to \{\pi_{o,n}E((X\backslash X_{i-1})\times^T \bbA^j\backslash\{0\})\}_{j\geq 1} \to \{\pi_{o,n}E((X\backslash X_i)\times^T \bbA^j\backslash\{0\})\}_{j\geq 1}  \to 0 $$
and, as a byproduct, the following exact sequences of abelian groups:
$$
\mathrm{lim}^1_{j} \pi_{o,n+1}E(Z_i\times^T \bbA^j\backslash\{0\}) \to \mathrm{lim}^1_{j} \pi_{o,n+1}E((X\backslash X_{i-1})\times^T \bbA^j\backslash\{0\}) \to \mathrm{lim}^1_{j} \pi_{o,n+1}E((X\backslash X_i)\times^T \bbA^j\backslash\{0\}) \to 0\,.
$$
Therefore, an inductive argument using the latter exact sequences of abelian groups and the fact that $\mathrm{lim}^1_{j \geq 1}\pi_{o,n+1}E(Z_i\times^T \bbA^j\backslash\{0\})=0$ for every $0 \leq i \leq m$ (proved in item (ii) of Step I), allows us to conclude that $\mathrm{lim}_{j\geq 1}^1 \pi_{o,n+1}E(X\times^T \bbA^j\backslash\{0\})=0$. Thanks to Milnor's short exact sequence \eqref{eq:Milnor-new2}, we hence obtain the right-hand side of \eqref{eq:completion}.

\begin{proposition}\label{prop:key1}
The above cofiber sequences \eqref{eq:cofiber-key} and \eqref{eq:seq-key2} are split in the homotopy category $\mathrm{Ho}(\cD)$.
\end{proposition}
\begin{proof}
We consider first the cofiber sequences \eqref{eq:cofiber-key}, i.e., the following cofiber sequences
\begin{eqnarray}\label{eq:sequences-main}
E([Z_i/T]) \stackrel{(\mathrm{i}_i)_\ast \circ \mathrm{q}_i^\ast}{\too} E([(X\backslash X_{i-1})/T]) \stackrel{\mathrm{j}^\ast_i}{\too} E([(X\backslash X_i)/T]) && 0 \leq i \leq m-1\,,
\end{eqnarray}
where $\mathrm{i}_i\colon W_i:= X_i\backslash X_{i-1}\hookrightarrow X\backslash X_{i-1}$ is a $T$-stable smooth closed subscheme, $\mathrm{j}_i\colon X\backslash X_i \hookrightarrow X\backslash X_{i-1}$ is the open complement of $W_i$, and $\mathrm{q}_i\colon W_i \to Z_i$ is a $T$-equivariant vector bundle. Recall that in a triangulated category (e.g., in the homotopy category $\mathrm{Ho}(\cD)$) a distinguished triangle $a \stackrel{f}{\to} b \stackrel{g}{\to} c \stackrel{\partial}{\to} a[1]$ is called {\em split} if $\partial=0$ or, equivalently, if there exists a morphism $c\stackrel{s}{\to} b$ such that $g \circ s=\id$. In this case, we have an induced isomorphism $(f,s)\colon a\oplus c \stackrel{\simeq}{\to} b$. In what follows, we will construct morphisms $\mathrm{s}_i, 0\leq i \leq m-1$, in the homotopy category $\Ho(\cD)$ making the following diagrams commute:
\begin{eqnarray}\label{eq:triangles}
\xymatrix{
E([(X\backslash X_{i-1})/T]) \ar[r]^-{\mathrm{j}_i^\ast} & E([(X\backslash X_i)/T]) \\
& E([Z_i/T])\oplus \bigoplus_{i' > i} E([Z_{i'}/T]) \ar@/^/[ul]^-{\mathrm{s}_i} \ar[u]_-{((\mathrm{i}_{i+1})_\ast \circ (\mathrm{q}_{i+1})^\ast, \mathrm{s}_{i+1})}^-\simeq\\
} && 0 \leq i \leq m-1
\end{eqnarray}
Note that, thanks to the equivalence $\mathrm{q}_m^\ast\colon E([Z_m/T])\to E([W_n/T])$ (recall that $X\backslash X_{m-1} = W_m$), an inductive argument using the commutative diagrams \eqref{eq:triangles} implies that the above cofiber sequences \eqref{eq:sequences-main} are split in the homotopy category $\Ho(\cD)$. Recall from Remark \ref{rk:key-remarks}(ii) that there exists an $\infty$-functor $\overline{E}$ (which does not necessarily preserve filtered colimits) making the following diagram commute:
\begin{equation*}
\xymatrix{
\dgcat(k)_\infty \ar[d]_-U \ar[rr]^-E && \cD \\
\mathrm{NMot}(k) \ar@/_1pc/[urr]_-{\overline{E}} && \,.
}
\end{equation*}
Therefore, since the induced functor $\Ho(\overline{E})$ is triangulated (in particular, it preserve (finite) direct sums), it suffices to construct morphisms $\mathrm{s}_i, 0 \leq i \leq m-1$, in $\Ho(\mathrm{NMot}(k))$ making the following diagrams commute:
\begin{eqnarray}\label{eq:triangles1}
\xymatrix{
U([(X\backslash X_{i-1})/T]) \ar[r]^-{\mathrm{j}_i^\ast} & U([(X\backslash X_i)/T]) \\
& U([Z_i/T])\oplus \bigoplus_{i' > i} U([Z_{i'}/T]) \ar@/^/[ul]^-{\mathrm{s}_i} \ar[u]_-{((\mathrm{i}_{i+1})_\ast \circ (\mathrm{q}_{i+1})^\ast,\mathrm{s}_{i+1})}
} && 0 \leq i \leq m-1\,;
\end{eqnarray} 
note that, in contrast with \eqref{eq:triangles}, the vertical morphisms in \eqref{eq:triangles1} are {\em not} invertible. By combining the filtration \eqref{eq:filtration} with Theorems \ref{thm:Gysin1} and \ref{thm:fibration}, we obtain the following cofiber sequences:
\begin{eqnarray}\label{eq:seq-A1}
U_{\bbA^1}([Z_i/T]) \stackrel{(\mathrm{i}_i)_\ast \circ \mathrm{q}_i^\ast}{\too} U_{\bbA^1}([(X\backslash X_{i-1})/T]) \stackrel{\mathrm{j}^\ast_i}{\too} U_{\bbA^1}([(X\backslash X_i)/T]) && 0 \leq i \leq m-1\,.
\end{eqnarray}
In particular, when $i=m-1$, we have the following cofiber sequence:
\begin{equation}\label{eq:seq-(n-1)}
U_{\bbA^1}([Z_{m-1}/T]) \stackrel{(\mathrm{i}_{m-1})_\ast \circ \mathrm{q}_{m-1}^\ast}{\too} U_{\bbA^1}([(X\backslash X_{m-2})/T]) \stackrel{\mathrm{j}^\ast_{m-1}}{\too} U_{\bbA^1}([(X\backslash X_{m-1})/T])\,.
\end{equation}
Since $X\backslash X_{m-1}=W_m$ and $T$ acts trivially on the projective $k$-scheme $Z_m$, the computation \eqref{eq:iso-equality-2-T} and the equivalence $\mathrm{q}_m^\ast\colon U_{\bbA^1}([Z_m/T]) \to U_{\bbA^1}([W_m/T])$ imply that \eqref{eq:seq-(n-1)} is split in the homotopy category $\Ho(\mathrm{NMot}_{\bbA^1}(k))$. As a consequence, there exists a morphism $\mathrm{s}_{m-1}$ making the following diagram commute:
$$
\xymatrix{
U([(X\backslash X_{m-2})/T]) \ar[rr]^-{\mathrm{j}_{m-1}^\ast} && U([(X\backslash X_{m-1})/T]) \\
&& U([Z_m/T]) \ar@/^/[ull]^-{s_{m-1}} \ar[u]_-{\mathrm{q}_m^\ast}^-\simeq\,.
}
$$
Now, a (similar) inductive argument, using the computation \eqref{eq:iso-equality-2-T}, allows us not only to conclude that the above cofiber sequences \eqref{eq:seq-A1} are split in $\Ho(\mathrm{NMot}_{\bbA^1}(k))$ but also to construct morphisms $\mathrm{s}_i, 0 \leq i \leq m-1$, in the homotopy category in $\Ho(\mathrm{NMot}_{\bbA^1}(k))$ making the following diagrams commute:
\begin{eqnarray}\label{eq:triangles2}
\xymatrix{
U_{\bbA^1}([(X\backslash X_{i-1})/T]) \ar[r]^-{\mathrm{j}_i^\ast} & U_{\bbA^1}([(X\backslash X_i)/T]) \\
& U_{\bbA^1}([Z_i/T])\oplus \bigoplus_{i' > i} U_{\bbA^1}([Z_{i'}/T]) \ar@/^/[ul]^-{\mathrm{s}_i} \ar[u]_-{((\mathrm{i}_{i+1})_\ast \circ (\mathrm{q}_{i+1})^\ast,\mathrm{s}_{i+1})}^-\simeq
} && 0 \leq i \leq m-1\,.
\end{eqnarray} 
Thanks to the computation \eqref{eq:iso-equality-T}, the proof follows now from the fact that the latter morphisms $\mathrm{s}_i$ belong also to the homotopy category $\Ho(\mathrm{NMot}(k))$. Moreover, the commutative diagrams \eqref{eq:triangles2} hold similarly in $\Ho(\mathrm{NMot}(k))$, i.e., the above commutative diagrams \eqref{eq:triangles1} hold. 

Finally, the case of the cofiber sequences \eqref{eq:seq-key2} is similar: simply replace the computations \eqref{eq:iso-equality-T}-\eqref{eq:iso-equality-2-T} by the computations \eqref{eq:iso-equality}-\eqref{eq:iso-equality1} and use the fact that $k$-schemes $Z_i \times^T \bbA^j\backslash \{0\}=Z_i \times \bbP^{j-1}$ are projective.
\end{proof}

\subsection*{Proof of item (iii) - case of a torus.} Similarly to item (iii) (case of a torus), we have isomorphisms:
\begin{equation}\label{eq:isos-last-key}
\pi_{o,n}E([X/T])/\langle (1-t_1)^{j_1} + \cdots + (1-t_r)^{j_r}\rangle \stackrel{\simeq}{\too} \pi_{o,n} \frac{E([X/T])}{(1-t_1)^{j_1} + \cdots + (1-t_r)^{j_r}} \quad j_1, \ldots, j_r \geq 1\,.
\end{equation}
Since the classical completion of abelian groups $(\pi_{o,n}E([X/T]))^\wedge_I$ may be described as the following limit $\mathrm{lim}_{j_1, \ldots, j_r \geq 1} \big( \pi_{o,n}E([X/T])/\langle (1-t_1)^{j_1} + \cdots + (1-t_r)^{j_r}\rangle\big)$, 
we hence obtain an induced isomorphism:
\begin{equation}\label{eq:induced-veryfinal}
(\pi_{o,n}E([X/T]))^\wedge_I \stackrel{\simeq}{\too}   \mathrm{lim}_{j_1, \ldots, j_r \geq 1} \pi_{o,n} \frac{E([X/T])}{(1-t_1)^{j_1} + \cdots + (1-t_r)^{j_r}}\,.
\end{equation}
Consider the multi-tower of abelian groups $\{\pi_{o,n} \frac{E([X/T])}{(1-t_1)^{j_1} + \cdots + (1-t_r)^{j_r}}\}_{j_1, \ldots, j_r\geq 1}$. Thanks to the above isomorphisms \eqref{eq:isos-last-key}, it follows from item (iii) (case of a circle) that for each choice of integers $j_1, \ldots, j_{i-1}, j_{i+1}, \ldots, j_r$ we have $\mathrm{lim}^1_{j_i \geq 1} \pi_{o, n+1} \frac{E([X/T])}{(1-t_1)^{j_1} + \cdots + (1-t_r)^{j_r}}=0$. This implies that the induced homomorphism
\begin{equation}\label{eq:induced-veryfinal1}
\pi_{o,n}(E([X/T])^\wedge_I) \stackrel{\simeq}{\too} \mathrm{lim}_{j_1, \ldots, j_r \geq 1} \pi_{o,n} \frac{E([X/T])}{(1-t_1)^{j_1} + \cdots + (1-t_r)^{j_r}}
\end{equation}
is invertible. The left-hand side of \eqref{eq:completion} is now defined as the composition of \eqref{eq:induced-veryfinal1} with the inverse~of~\eqref{eq:induced-veryfinal}.

Now, consider the multi-tower of abelian groups $\{\pi_{o,n}E(X\times^T (\bbA^{j_1}\backslash\{0\} \times \cdots \times \bbA^{j_r}\backslash\{0\}))\}_{j_1, \ldots, j_r \geq 1}$. It follows from item (iii) (case of a circle) that $\mathrm{lim}^1_{j_i \geq 1} \pi_{o,n+1}E(X\times^T (\bbA^{j_1}\backslash\{0\} \times \cdots \times \bbA^{j_r}\backslash\{0\}))=0$ for each choice of integers $j_1, \ldots, j_{i-1}, j_{i+1}, \ldots, j_r$. This implies that the induced homomorphism
\begin{equation}\label{eq:induced-veryfinal11}
\pi_{o,n} E_T(X) \stackrel{\simeq}{\too} \mathrm{lim}_{j_1, \ldots, j_r \geq 1} \pi_{o,n} E(X\times^T (\bbA^{j_1}\backslash\{0\} \times \cdots \times \bbA^{j_r}\backslash\{0\}))
\end{equation}
is invertible. The right-hand side of \eqref{eq:completion} is now defined as the composition of \eqref{eq:induced-veryfinal11} with the natural identification of the right-hand side of \eqref{eq:induced-veryfinal11} with $\mathrm{lim}_{j\geq 1} \pi_{o,n} E(X\times^T (\bbA^j\backslash\{0\})^r)$.
\section{Proof of Theorem \ref{thm:reduction}}
We start with two preliminary results of independent interest. 
\begin{proposition}[Equivariant projective bundle theorem]\label{prop:bundle}
Let $X$ be a separable $k$-scheme of finite type equipped with a $G$-action, $V \to X$ a $G$-vector bundle on $X$ of rank $r$, and $\pi\colon \bbP(V) \to X$ the associated projective bundle ($\bbP(V)$ is equipped with an induced $G$-action and the map $\pi$ is $G$-equivariant). Given an $\infty$-functor $E\colon \dgcat(k)_\infty \to \cD$ satisfying conditions (C1) and (C3), the following dg functors
\begin{eqnarray*}
\Phi_i\colon \perf_\dg([X/G]) \too \perf_\dg([\bbP(V)/G]) & \cF \mapsto \pi^\ast(\cF) \otimes \cO_{\bbP(V)}(i) & 0 \leq i \leq r-1
\end{eqnarray*}
give rise to an equivalence of $K([\bullet/G])$-modules $E([X/G])^{\oplus r} \to E([\bbP(V)/G])$.
\end{proposition}
\begin{proof}
Recall first from Remark \ref{rk:key-remarks}(ii) that since $E$ satisfies conditions (C1) and (C3), there exists an $\infty$-functor $\overline{E}$ (which does not necessarily preserve filtered colimits) making the following diagram commute:
\begin{equation}\label{eq:factorization-new}
\xymatrix{
\dgcat(k)_\infty \ar[d]_-U \ar[rr]^-E && \cD \\
\mathrm{NMot}(k) \ar@/_1pc/[urr]_-{\overline{E}} && \,.
}
\end{equation}
As proved by Elagin in \cite[Thm.~10.1]{Elagin}, the triangulated functors $\dgHo(\Phi_i)$, $0 \leq i \leq r-1$, are fully-faithful and give rise to the following semi-orthogonal decomposition:
$$ \perf([\bbP(V)/G])=\langle \dgHo(\Phi_0)(\perf([X/G])), \ldots, \dgHo(\Phi_{r-1})(\perf([X/G])) \rangle\,.$$
Consequently, following \cite[\S2.1 and \S8.3-\S8.4]{book}, we obtain an induced equivalence of $K([\bullet/G])$-modules $U([X/G])^{\oplus r} \to U([\bbP(V)/G])$. The proof follows now from the above commutative diagram \eqref{eq:factorization-new}.
\end{proof}
\begin{proposition}[Strong homotopy invariance property]\label{prop:strong}
Let $X$ be a smooth separable $k$-scheme of finite type equipped with a $G$-action, $V \to X$ a $G$-vector bundle, and $f\colon Y \to X$ a torsor under $V$. Assume that $G$ acts on $Y$ and that the maps $f$ and $V \times_X Y \to Y$ are $G$-equivariant. Given an $\infty$-functor $E\colon \dgcat(k)_\infty \to \cD$ satisfying conditions (C1), (C3) and (C4), we have an induced equivalence~of~$K([\bullet/G])$-modules:
\begin{equation}\label{eq:morphism-induced}
f^\ast\colon E([X/G]) \too E([Y/G])\,.
\end{equation}
\end{proposition}
\begin{proof}
Following Thomason \cite[Thm.~4.1]{ThomasonActions}, there exists a short exact sequence of $G$-vector bundles on $X$
$$ 0 \too V \too W \stackrel{\varphi}{\too} \bbA^1_X \too 0\,,$$
where $\bbA^1_X$ stands for the trivial line bundle (with trivial $G$-action), such that $\varphi^{-1}(1)\simeq Y$. Consequently, the smooth $k$-scheme $Y$ may be identified with the open complement of the $G$-stable closed immersion $\mathrm{i}\colon \bbP(V) \hookrightarrow \bbP(W)$. Therefore, similarly to Thomason's proof, by combining Proposition \ref{prop:bundle} with the fact that the $\infty$-functor $E$ satisfies condition (C4), we conclude that \eqref{eq:morphism-induced} is an equivalence of $K([\bullet/G])$-modules.
\end{proof}
Let $T$ be a $k$-split maximal torus of $G$ and $B$ a Borel subgroup of $G$ containing $T$. The proof of Theorem \ref{thm:reduction} follows now from the following result:
\begin{proposition}
Let $X$ be a smooth separated $k$-scheme of finite type equipped with a $G$-action, and $E\colon \dgcat(k)_\infty \to \cD$ a localizing $\bbA^1$-homotopy invariant satisfying the extra condition (C4). 
\begin{itemize}
\item[(i)] We have the following commutative diagram of $K([\bullet/G])$-modules with $\mathrm{ind}\circ \mathrm{res} =\id$:
\begin{eqnarray}\label{eq:diagram1}
\xymatrix{
E([X/G]) \ar[r]^-{\mathrm{res}} \ar[d]_-{\eqref{eq:induced-G}} & E([X/B]) \ar[d]_-{\eqref{eq:induced-G}} \ar[r]^-{\mathrm{ind}} & E([X/G]) \ar[d]_-{\eqref{eq:induced-G}}\\
E_G(X) \ar[r]_-{\mathrm{res}}  & E_B(X) \ar[r]_-{\mathrm{ind}} & E_G(X) 
} \,.
\end{eqnarray}
\item[(ii)] We have the following commutative square of $K([\bullet/G])$-modules (with $\mathrm{res}$ an equivalence):
\begin{eqnarray}\label{eq:diagram2}
\xymatrix{
E([X/B]) \ar[d]_-{\eqref{eq:induced-G}} \ar[r]^-{\mathrm{res}} & E([X/T]) \ar[d]^-{\eqref{eq:morphism}} \\
E_B(X) \ar[r]_-{\mathrm{res}} & E_T(X)
}
\end{eqnarray}
\end{itemize}
\end{proposition}
\begin{proof}
We start by proving item (i). Consider the projective homogeneous variety $G/B$ and the $G$-equivariant projection map $\pi\colon G/B \times X \to X$. Since $\pi$ is flat and proper, it yields the following adjunction:
\begin{equation*}
\xymatrix{
\perf([(G/B\times X)/G]) \ar@<1ex>[d]^{\pi_\ast}\\
\perf([X/G]) \ar@<1ex>[u]^{\pi^\ast} \,.
}
\end{equation*}
We have natural isomorphisms $(\pi_\ast \circ \pi^\ast)(\cF) \stackrel{\mathrm{(a)}}{\simeq} \cF \otimes (\pi_\ast \circ \pi^\ast)(\cO)\simeq \cF \otimes \pi_\ast(\cO) \stackrel{\mathrm{(b)}}{\simeq} \cF \otimes \cO \simeq \cF$ for every $\cF \in \perf([X/G])$, where $\mathrm{(a)}$ follows from the projection formula and $\mathrm{(b)}$ from the classical Kempf vanishing theorem for $G/B$. Consequently, we obtain the following morphisms of $K([\bullet/G])$-modules with $\pi_\ast \circ \pi^\ast =\id$:
\begin{eqnarray}\label{eq:natural-last}
E([X/G])\stackrel{\pi^\ast}{\too} E([(G/B\times X)/G]) \stackrel{\pi_\ast}{\too} E([X/G])\,.
\end{eqnarray}
By combining \eqref{eq:natural-last} with Lemma \ref{lem:Morita} (with $H=B$), we hence obtain the upper-part of the diagram \eqref{eq:diagram1}. Now, let $\{(V_j,U_j)\}_{j\geq 1}$ be an admissible gadget for $G$ (and hence for $B$). By replicating the above argument with $X$ replaced by $X\times U_j$, we obtain morphisms of $K([\bullet/G])$-modules with $\mathrm{ind}\circ \mathrm{res} =\id$:
\begin{eqnarray}\label{eq:retraction}
E([(X\times U_j)/G]) \stackrel{\mathrm{res}}{\too} E([(X\times U_j)/B]) \stackrel{\mathrm{ind}}{\too} E([(X\times U_j)/G]) && j \geq 1\,.
\end{eqnarray}
Consequently, by applying the $\infty$-functor $\mathrm{lim}_{j \geq 1}(-)$ to \eqref{eq:retraction}, we obtain the bottom-part of the diagram \eqref{eq:diagram1}. Finally, the commutativity of \eqref{eq:diagram1} follows now from the definition of the morphism \eqref{eq:induced-G}.

We now prove item (ii). Consider the characteristic filtration $ \{1\} \subseteq B^u_n \subseteq B^u_{n-1} \subseteq \cdots \subseteq B_1^u \subseteq B_0^u = B^u$ of the unipotent radical $B^u$ of $B$. Let us write $T\!B^u_i$ for the subgroup of $B$ generated by $T$ and $B^u_i$. Following Thomason \cite[Thm.~1.13]{Thomason2}, for every $1\leq i \leq n$, we have a $B$-vector bundle $B^u_{i-1}/B^u_i \to B/T\!B^u_{i-1}$ and a torsor $f_i \colon B/T\!B^u_i \to B/T\!B^u_{i-1}$ under $B^u_{i-1}/B^u_i$. Consequently, by applying Proposition \ref{prop:strong} (with $G$ replaced by $B$ and $X$ replaced by $B^u_{i-1}/B^u_i \times X$), we obtain equivalences of $K([\bullet/B])$-modules:
\begin{eqnarray}
(f_i\times \id)^\ast \colon E([(B/T\!B^u_{i-1}\times X)/B]) \too E([(B/T\!B^u_i \times X)/B])&& 1\leq i \leq n\,.
\end{eqnarray}
Their composition yields then an equivalence of $K([\bullet/B])$-modules:
\begin{equation}\label{eq:equivalence-final}
E([X/B]) \too E([(B/T\times X)/B])\,.
\end{equation}
By combining \eqref{eq:equivalence-final} with Lemma \ref{lem:Morita} (with $G=B$ and $H=T$), we hence obtain the upper-part of the square \eqref{eq:diagram2}. Now, let $\{(V_j,U_j)\}_{j \geq 1}$ be an admissible gadget for $B$ (and hence for $T$). By replicating the above argument with $X$ replaced by $X\times U_j$, we obtain equivalences of $K([\bullet/B])$-modules
\begin{eqnarray}\label{eq:restriction-last}
E([(X\times U_j)/B]) \stackrel{\mathrm{res}}{\too} E([(X\times U_j)/T]) && j \geq 1\,.
\end{eqnarray}
Consequently, by applying the $\infty$-functor $\mathrm{lim}_{j\geq 1}(-)$ to \eqref{eq:restriction-last}, we obtain the bottom part of the square \eqref{eq:diagram2}. Finally, the commutativity of \eqref{eq:diagram2} follows now from the definition of the morphisms \eqref{eq:morphism} and \eqref{eq:induced-G}.
\end{proof}
\begin{lemma}[Morita equivalence]\label{lem:Morita}
Let $H \subset G$ be a closed subgroup $k$-scheme and $X$ a smooth separated $k$-scheme of finite type equipped with a $G$-action. Given an $\infty$-functor $E\colon \dgcat(k)_\infty \to \cD$ satisfying conditions (C1) and (C3), we have an equivalence of $K([\bullet/G])$-modules:
\begin{equation}\label{eq:equivalence-last}
E([(G/H \times X)/G]) \too E([X/H])\,.
\end{equation}
\end{lemma}
\begin{proof}
Let us write $G\times^H X$ for the quotient of $G\times X$ by the $H$-action $h(g,x):=(gh^{-1}, hx)$. As proved by Thomason in \cite[Thm.~1.10]{Thomason2}, restriction along the $H$-equivariant map $X=H \times^H X \to G\times^H X$ gives rise to a Morita equivalence $ \perf_\dg([(G\times^H X)/G]) \to \perf_\dg([X/H])$. Consequently, the proof follows now from the isomorphism $G\times^H X \stackrel{\simeq}{\to} G/H \times X, (g,x) \mapsto (gH,gx)$.
\end{proof}
\begin{remark}[Generalization]
Note that if one ignores the $K([\bullet/G])$-module structure, then Lemma \ref{lem:Morita} holds more generally for every $\infty$-functor $E\colon \dgcat(k)_\infty \to \cD$.
\end{remark}
\appendix

\section{Homotopy $K$-theory of quotient stacks}
The following result, which is of independent interest, extends a previous result of Weibel \cite{Weibel} on homotopy $K$-theory from the realm of schemes to the broad setting of quotient stacks.
\begin{theorem}\label{thm:homotopy}
\begin{itemize}
\item[(i)] Let $X$ be a smooth quasi-compact separated $k$-scheme equipped with a $T$-action. Under these assumptions, the canonical morphism $\bbK([X/T])\to KH([X/T])$ is an equivalence of spectra.
\item[(ii)] Let $X$ be a smooth separated $k$-scheme of finite type equipped with a $G$-action. Under these assumptions, the canonical morphism $\bbK([X/G])\to KH([X/G])$ is an equivalence of spectra.
\end{itemize}
\end{theorem}
\begin{proof}
We start by proving item (i). Recall from \S\ref{sub:K-theory} that $KH([X/T]):=\mathrm{colim}_m \bbK(\perf_\dg([X/T])\otimes \Delta_m)$. Following Bass \cite[\S XII]{Bass}, consider the following abelian groups (defined recursively)
\begin{eqnarray}\label{eq:Bass}
& N^p\bbK_q([X/T]):=\mathrm{kernel}\big(N^{p-1}\bbK_q(\perf_\dg([X/T])[t]) \stackrel{t=0}{\too} N^{p-1}\bbK_q(\perf_\dg([X/T]))\big) & p\geq 0\,\,\,\, q\in \bbZ
\end{eqnarray}
with $N^0 \bbK_q([X/T]):=\bbK_q([X/T])$. We claim that $N^p\bbK_q([X/T])=0$ for every $p\geq 0$ and $q\in \bbZ$. Note that thanks to the standard convergent right half-plane spectral sequence 
$$E^1_{pq}=N^p\bbK_q([X/T]) \Rightarrow KH_{p+q}([X/T])$$ associated to the simplicial spectrum $m \mapsto \bbK(\perf_\dg([X/T])\otimes \Delta_m)$, this claim would imply that the edge morphisms $\bbK_q([X/T]) \to KH_q([X/T])$ of the spectral sequence are invertible and, as a consequence, that the canonical morphism $\bbK([X/T]) \to KH([X/T])$ is an equivalence of spectra. Thanks to the above definition \eqref{eq:Bass}, in order to prove our claim it suffices to show that the following canonical homomorphisms
\begin{eqnarray}\label{eq:canonical}
\bbK_q(\perf_\dg([X/T])[t_1, \ldots, t_m])\too \bbK_q(\perf_\dg([X/T])[t_1, \ldots, t_m][t]) && q \in \bbZ
\end{eqnarray}
are invertible. Under the following Morita equivalences (consult Lemma \ref{lem:A1}(i))
$$
\perf_\dg([X/T])[t_1, \ldots, t_m]  \to   \perf_\dg([(X\times \bbA^m)/T]) \quad
\perf_\dg([X/T)[t_1, \ldots, t_m][t]  \to \perf_\dg([(X\times \bbA^m\times \bbA^1)/T])\,,
$$
where $T$ acts trivially on $\bbA^m$ and on $\bbA^1$, the homomorphisms \eqref{eq:canonical} correspond to the homomorphisms
\begin{eqnarray*}
\pi_{\bbA^1}^\ast\colon \bbK_q([(X\times \bbA^m)/T]) \too \bbK_q([(X\times \bbA^m \times \bbA^1)/T]) && q \in \bbZ
\end{eqnarray*}
induced by the projection $\pi_{\bbA^1}\colon X\times \bbA^m \times \bbA^1 \to X\times \bbA^m$. Therefore, it suffices to show that $\pi_{\bbA^1}^\ast$ is invertible. Consider the $k$-scheme $\bbP^1_{X\times \bbA^m}:=X\times \bbA^m \times \bbP^1$ equipped with the induced $T$-action ($T$ acts trivially on $\bbA^m$ and on $\bbP^1$), the $T$-stable closed subscheme $\mathrm{i}\colon X\times \bbA^m = X\times \bbA^m \times \{0\} \hookrightarrow X\times \bbA^m \times \bbP^1$, and the open complement $\mathrm{j}\colon X\times \bbA^m \times \bbA^1 \hookrightarrow X\times \bbA^m \times \bbP^1$ of $X\times \bbA^m$. Since $\bbP^1_{X\times \bbA^m}$ and $X\times \bbA^1$ are smooth, Thomason's localization theorem \cite[Thm.~2.7]{ThomasonActions} yields the following long exact sequence of abelian groups:
\begin{equation}\label{eq:long}
\cdots \too \bbK_q([(X\times \bbA^m)/T]) \stackrel{\mathrm{i}_\ast}{\too} \bbK_q([(\bbP^1_{X\times \bbA^m})/T]) \stackrel{\mathrm{j}^\ast}{\too} \bbK_q([(X\times \bbA^m\times \bbA^1)/T]) \too \cdots 
\end{equation}
Moreover, thanks to Thomason's projective bundle theorem \cite[Thm.~3.11]{ThomasonActions}, we have the isomorphisms 
\begin{eqnarray}\label{eq:iso-bundle}
(\iota_0, \iota_{-1})\colon K_q([(X\times \bbA^m)/T])^{\oplus 2} \stackrel{\simeq}{\too} \bbK_q([(\bbP^1_{X\times \bbA^m})/T]) && q\in \bbZ
\end{eqnarray}
induced by the following fully-faithful dg functors
\begin{eqnarray*}
\iota_0\colon \perf_\dg([(X\times \bbA^m)/T]) \too \perf_\dg([(\bbP^1_{X\times \bbA^m})/T]) && \cF \mapsto \pi_{\bbP^1}^\ast(\cF) \\
\iota_{-1}\colon \perf_\dg([(X\times \bbA^m)/T]) \too \perf_\dg([(\bbP^1_{X\times \bbA^m})/T]) && \cF \mapsto \pi_{\bbP^1}^\ast(\cF)\otimes \cO(-1)\,,
\end{eqnarray*}
where $\pi_{\bbP^1}\colon X\times \bbA^m \times \bbP^1 \to X\times \bbA^m$ stands for the projection. Under the isomorphisms \eqref{eq:iso-bundle}, the above homomorphism $\mathrm{i}_\ast$ in \eqref{eq:long} agrees with the difference $\iota_0 - \iota_{-1}$ and both compositions $\mathrm{j}^\ast\circ \iota_0$ and $\mathrm{j}^\ast\circ \iota_{-1}$ agree with $\pi^\ast_{\bbA^1}$. Therefore, making use of Lemma \ref{lem:aux}, the above long exact sequence \eqref{eq:long} breaks up into the following short exact sequences of abelian groups:
\begin{equation}\label{eq:long1}
\xymatrix@C=1em@R=2em{
0 \too \bbK_q([(X\times \bbA^m)/T]) \ar[r]^-{\mathrm{i}_\ast} & \bbK_q([(\bbP^1_{X\times \bbA^m})/T]) \ar[r]^-{\mathrm{j}^\ast} & \bbK_q([(X\times \bbA^m\times \bbA^1)/T]) \too 0 \\
0 \too \bbK_q([(X\times \bbA^m)/T]) \ar@{=}[u] \ar[r]& \bbK_q([(X\times \bbA^m)/T])^{\oplus 2} \ar[r] \ar[u]_-{(\iota_0-\iota_{-1}, \iota_0)}^-\simeq & \bbK_q([(X\times \bbA^m)/T]) \too 0 \ar[u]_-{\pi^\ast_{\bbA^1}} \,.
}
\end{equation}
Finally, since the middle vertical morphism in \eqref{eq:long1} is invertible, we conclude that $\pi^\ast_{\bbA^1}$ is also invertible. 

The proof of item (ii) is similar. Note that since $X$ is smooth, separated and of finite type, then $X$ is also regular, separated and Noetherian. This implies that $X$ admits an ample family of (non-equivariant) line bundles. Consequently, using Remark \ref{rk:general}, it suffices to replace Lemma \ref{lem:A1}(i) by Lemma \ref{lem:A1}(ii).
\end{proof}
\begin{lemma}\label{lem:A1} 
\begin{itemize}
\item[(i)] Let $X$ be a quasi-compact separated normal $k$-scheme equipped with a $T$-action, and $\bbA^1:=\mathrm{Spec}(k[t])$ the affine line (with trivial $T$-action). Under these assumptions, the dg functor
\begin{eqnarray}\label{eq:Morita-proj}
\perf_\dg([X/T])[t] \too \perf_\dg([(X\times \bbA^1)/T]) && \cF \mapsto \pi^\ast(\cF)\,,
\end{eqnarray}
induced by the projection $\pi\colon X\times \bbA^1 \to X$, is a Morita equivalence.
\item[(ii)] Let $X$ be a separated normal $k$-scheme of finite type equipped with a $G$-action, and $\bbA^1:=\mathrm{Spec}(k[t])$ the affine line. If the quotient stack $[X/G]$ has the resolution property, then the dg functor
\begin{eqnarray}\label{eq:Morita-proj1}
\perf_\dg([X/G])[t] \too \perf_\dg([(X\times \bbA^1)/G]) && \cF \mapsto \pi^\ast(\cF)\,,
\end{eqnarray}
induced by the projection $\pi\colon X\times \bbA^1 \to X$, is a Morita equivalence.
\end{itemize}
\end{lemma}
\begin{proof}
We start by proving item (i). Consider the following adjunction:
\begin{equation*}
\xymatrix{
\cD_{\mathrm{Qcoh}}([(X\times \bbA^1)/T]) \ar@<1ex>[d]^{\pi_\ast}\\
\cD_{\mathrm{Qcoh}}([X/T]) \ar@<1ex>[u]^{\pi^\ast} \,.
}
\end{equation*}
Thanks to Theorem \ref{thm:compact}, the triangulated category $\cD_{\mathrm{Qcoh}}([X/T])$ admits a set of perfect (=compact) generators $\{\cG_i\}_{i \in I}$. Therefore, since the functor $\pi^\ast$ preserve perfect objects and the functor $\pi_\ast$ preserve arbitrary direct sums and is moreover conservative, we conclude that $\{\pi^\ast(\cG_i)\}_{i\in I}$ is a set of perfect (and hence compact) generators of $\cD_{\mathrm{Qcoh}}([(X\times \bbA^1)/T])$. Moreover, we have the following natural identifications:
\begin{eqnarray*}
{\bf R} \mathrm{Hom}(\pi^\ast(\cG_i), \pi^\ast(\cG_{i'}))\simeq {\bf R} \mathrm{Hom}(\cG_i, \pi_\ast \pi^\ast(\cG_{i'})) \simeq {\bf R} \mathrm{Hom}(\cG_i, \cG_{i'}[t]) \stackrel{\mathrm{(a)}}{\simeq} {\bf R}\mathrm{Hom}(\cG_i, \cG_{i'})[t] && i, i' \in I\,,
\end{eqnarray*}
where (a) follows from the compactness of $\cG_i$. This hence implies that \eqref{eq:Morita-proj} is a Morita equivalence.

The proof of item (ii) is similar: simply replace Theorem \ref{thm:compact} by Remark \ref{rk:general}.
\end{proof}

\begin{lemma}\label{lem:aux}
Given an isomorphism $(f,g)\colon A \oplus A \stackrel{\simeq}{\to} B$ in an additive category, the modified morphism $(f-g,f)\colon A \oplus A \to B$ is also invertible.
\end{lemma}
\begin{proof}
A simple exercise that we leave for the reader.
\end{proof}

\section{Semi-topological $K$-theory}\label{sec:semi}
The following properties of semi-topological $K$-theory are of independent interest.
\begin{theorem}\label{prop:semi}
The $\infty$-functor \eqref{eq:semi} satisfies conditions $\mathrm{(C1)}\text{-}\mathrm{(C2)}\text{-}\mathrm{(C3)}\text{-}\mathrm{(C4)}$.
\end{theorem}
\begin{remark}[Real semi-topological $K$-theory]
Theorem \ref{prop:semi} holds {\em mutatis mutandis} for real semi-topological $K$-theory; consult Remark \ref{rk:real}.
\end{remark}
\begin{proof}
Condition (C1) follows from the fact that the $\infty$-functor $-\otimes \cA$ preserve short exact sequences of dg categories and the $\infty$-functor $\bbK(-)$ satisfies condition (C1).

Let us now prove condition (C2). Similarly to \S\ref{sub:K-theory}, consider the co-simplicial affine $\bbC$-scheme $\Delta^m:=\mathrm{Spec}(\bbC[t_0, \ldots, t_m]/\langle \sum_{i=0}^m t_i -1\rangle)$, with $m \geq 0$. Let us write $\pi\colon \Delta^1 \to \Delta^0$ for the projection, $\iota_0\colon \Delta^0 \to \Delta^1$ for the closed embedding $\{0\}$, and $\mu\colon \Delta^1 \times \Delta^1 \to \Delta^1$ for the multiplication. Note that since the dg category $k[t]$ is Morita equivalent to $\perf_\dg(\Delta^1)$, it suffices to show that $\id \otimes \pi^\ast \colon K^{\mathrm{st}}(\cA) \to K^{\mathrm{st}}(\cA\otimes \perf_\dg(\Delta^1))$ is an equivalence. Consider the morphism $\id \otimes \iota_0^\ast \colon K^{\mathrm{st}}(\cA\otimes \perf_\dg(\Delta^1)) \to K^{\mathrm{st}}(\cA)$. Note that since $\pi \circ \iota_0 = \id$, the composition $(\id \otimes \iota_0^\ast)\circ (\id \otimes \pi^\ast)$ is equal to the identity. We claim that the composition $(\id \otimes \pi^\ast) \circ (\id \otimes \iota_0^\ast)$ is also equal to the identity (in the homotopy category $\Ho(\Spt_\infty)$); note that this claim would automatically imply that $\id \otimes \pi^\ast$ is an equivalence. In order to prove our claim, consider the linear maps (with $0 \leq j \leq m$)
\begin{eqnarray*}
\Delta^{m+1}\stackrel{\varphi_j}{\too} \Delta^m \times \Delta^1 &\mathrm{with}& \varphi_j(v_i):=\begin{cases} v_i \times 0 & \mathrm{if}\,\, i\leq j \\ v_{i-1} \times 1 & \mathrm{if}\,\, i> j \\ \end{cases}\,,
\end{eqnarray*}
where $v_i=(0, \ldots, 1, \ldots, 0)$ stands for the $i^{\mathrm{th}}$ vertex. Making use of them, we can construct morphisms of spectra $h_j \colon K^{\mathrm{st}}(\cA\otimes \perf_\dg(\Delta^1))_m \to K^{\mathrm{st}}(\cA\otimes \perf_\dg(\Delta^1))_m$, with $0 \leq j \leq m$, by the following recipe: the map of topological spaces $f_V\colon \Delta^m_{\mathrm{top}} \to V^{\mathrm{an}}$ is sent to the composition 
$$\Delta^{m+1}_{\mathrm{top}} \stackrel{\varphi_j^{\mathrm{an}}}{\too} \Delta^m_{\mathrm{top}}\times \Delta^1_{\mathrm{top}} \stackrel{f_V \times \mathrm{inc}}{\too} (V \times \Delta^1)^{\mathrm{an}}$$ and the corresponding morphism of spectra is given by the composition of  
$$ \id \otimes \id \otimes \mu^\ast \colon \bbK(\perf_\dg(V) \otimes \cA \otimes \perf_\dg(\Delta^1)) \too \bbK(\perf_\dg(V) \otimes \cA \otimes \perf_\dg(\Delta^1 \times \Delta^1))$$
with the following natural identifications
\begin{eqnarray}
\bbK(\perf_\dg(V) \otimes \cA \otimes \perf_\dg(\Delta^1 \times \Delta^1)) &\simeq & \bbK(\perf_\dg(V) \otimes \cA \otimes \perf_\dg(\Delta^1) \otimes \perf_\dg(\Delta^1)) \label{eq:a} \\
& \simeq & \bbK(\perf_\dg(V \times \Delta^1) \otimes \cA \otimes \perf_\dg(\Delta^1)) \label{eq:b}\,,
\end{eqnarray}
where \eqref{eq:a}-\eqref{eq:b} follow from the Morita equivalences (consult \cite[Lem.~4.26]{Gysin}):
\begin{eqnarray*}
\perf_\dg(Y) \otimes \perf_\dg(\Delta^1) \to \perf_\dg(Y \times \Delta^1) \quad (\cF, \cG) \mapsto \cF\boxtimes \cG &\mathrm{with} & Y=\Delta^1 \,\,\mathrm{or}\,\,Y=V\,.
\end{eqnarray*}
Now, a simple verification shows that the assignment $m \mapsto \{h_j\}_{0 \leq j \leq m}$ is a simplicial homotopy between the composition $(\id \otimes \pi^\ast)\circ (\id \otimes \iota_0^\ast)$ and the identity (considered as endomorphisms of the simplicial spectrum $m \mapsto K^{\mathrm{st}}(\cA\otimes \perf_\dg(\Delta^1))_m$). By definition of semi-topological $K$-theory, this implies our claim.

Condition (C3) follows from the fact that the $\infty$-functors $-\otimes \cA$ and $\bbK(-)$ preserve filtered colimits.

Finally, let us prove condition (C4). Let $X$ be a smooth separated $\bbC$-scheme of finite type equipped with a $G$-action, $\mathrm{i}\colon Z \hookrightarrow X$ a $G$-stable smooth closed subscheme, and $\mathrm{j}\colon U \hookrightarrow X$ the open complement of $Z$. Given a smooth separated $\bbC$-scheme of finite type $V$, it follows from the work of Thomason \cite[Thms.~2.7 and 5.7]{ThomasonActions} that we have an induced cofiber sequence of $K([\bullet/G])$-modules:
\begin{equation}\label{eq:cof1}
\bbK([(V\times Z)/G]) \stackrel{(\id \times \mathrm{i})_\ast}{\too} \bbK([(V\times X)/G]) \stackrel{(\id \times \mathrm{j})^\ast}{\too} \bbK([(V\times U)/G])\,.
\end{equation}
Moreover, thanks to the following Morita equivalences (obtained by replicating the proof of \cite[Lem.~4.26]{Gysin})
\begin{eqnarray*}
\perf_\dg(V) \otimes \perf_\dg([W/G]) \too \perf_\dg([(V\times W)/G]) & (\cF,\cG) \mapsto \cF\boxtimes \cG & \mathrm{with}\quad W=Z, X, U
\end{eqnarray*}
the cofiber sequence \eqref{eq:cof1} identifies with the following cofiber sequence of $K([\bullet/G])$-modules:
\begin{equation}\label{eq:cof2}
\bbK\left(\perf_\dg(V)\otimes\left(\perf_\dg([Z/G]) \stackrel{\mathrm{i}_\ast}{\too} \perf_\dg([X/G]) \stackrel{\mathrm{j}^\ast}{\too} \perf_\dg([U/G])\right)\right)\,.
\end{equation}
Now, recall from \S\ref{sub:semi} that semi-topological $K$-theory is defined as follows:
\begin{eqnarray}\label{def:verylast}
K^{\mathrm{st}}(-)\colon \dgcat(\bbC)_\infty \too \mathrm{Spt}_\infty && \cA \mapsto \mathrm{colim}_m K^{\mathrm{st}}(\cA)_m \,,
\end{eqnarray}
where $K^{\mathrm{st}}(\cA)_m:=\mathrm{colim}_{(V,f_V) \in \Delta^{m, \downarrow}_{\mathrm{top}}} \,\bbK(\perf_\dg(V) \otimes \cA)$. Following \cite[Prop.~3.22]{Blanc}, we can assume without loss of generality that the affine $\bbC$-schemes $V$ are smooth. Consequently, since cofiber sequences are stable under colimits, the proof follows now from the combination of \eqref{eq:cof2} and \eqref{def:verylast}.
\end{proof}

\subsection*{Acknowledgments.} The authors are grateful to the Institut des Hautes \'Etudes Scientifiques (IH\'ES) for its hospitality and excellent working conditions.


\begin{thebibliography}{00}

\bibitem{AH} M.~Atiyah and F.~Hirzebruch, {\em Vector bundles and homogeneous spaces}. Proc. Sympos. Pure Math., Vol. {\bf III}, AMS, Providence, R.I. (1961), 7--38.

\bibitem{AS} M.~Atiyah and G.~Segal, {\em Equivariant $K$-theory and completion}. J. Differential Geometry {\bf 3} (1969), 1--18.

\bibitem{Bass} H.~Bass, {\em Algebraic $K$-theory}. W. A. Benjamin, New York, 1968.

\bibitem{Beilinson} A. Beilinson, {\em Coherent sheaves on $\bbP^n$ and problems in linear algebra}. Funktsional. Anal. i Prilozhen. {\bf 12} (1978), no.~3, 68--69.

\bibitem{BB} A.~Bialynicki-Birula, {\em Some theorems on actions of algebraic groups}. Ann. Math. {\bf 98} (1973), no.~2, 480--497.

\bibitem{Blanc} A.~Blanc, {\em Topological $K$-theory of complex noncommutative spaces}. Compos. Math. {\bf 152} (2016), no.~3, 489--555.

\bibitem{BGT} A.~Blumberg, D.~Gepner and G.~Tabuada, {\em A universal characterization of higher algebraic $K$-theory}. Geom. Topol. {\bf 17} (2013), no.~2, 733--838.

\bibitem{BV} A.~Bondal and M.~Van den Bergh, {\em Generators and representability of functors in commutative and noncommutative geometry}. Mosc. Math. J. {\bf 3} (2003), no. 1, 1--36, 258.

\bibitem{Brion} M.~Brion, {\em Equivariant Chow groups for torus actions}. Transform. Groups {\bf 2} (1997), no.~3, 225--267.

\bibitem{Browder} W.~Browder, {\em Algebraic $K$-theory with coefficients $\bbZ/p$}. Geometric applications of homotopy theory (Proc. Conf., Evanston, Ill., 1977), {\bf I}, Lecture Notes in Math., vol. {\bf 657}, pp. 40--84.

\bibitem{Carlsson} G.~Carlsson, {\em Derived completions in stable homotopy theory}. J. Pure Appl. Algebra, {\bf 212}, (2008), no.~3, 550--577.

\bibitem{CJ} G.~Carlsson and R.~Joshua, {\em Atiyah-Segal derived completions for equivariant algebraic $G$-theory and $K$-theory}. Available at arXiv:1906.06827.

\bibitem{Drinfeld} V.~Drinfeld, {\em DG quotients of DG categories}. J. Algebra {\bf 272} (2004), no.~2, 643--691.

\bibitem{EG} D.~Edidin and W.~Graham, {\em Riemann-Roch for equivariant Chow groups}. Duke Math.~J. {\bf 102}(3) (2000), 567--594.

\bibitem{Elagin} A.~Elagin, {\em Descent theory for semiorthogonal decompositions}. Available at arXiv:1206.2881.

\bibitem{Feigin-Tsygan} B.~Feigin and B.~Tsygan, {\em Additive $K$-theory, $K$-theory, arithmetic and geometry (Moscow, 1984--1986)}. Lecture Notes in Math., vol. {\bf 1289}, Springer, Berlin, 1987, pp. 67--209.

\bibitem{FW2} E.~Friedlander, C.~Haesemeyer and M.~Walker, {\em Techniques, computations, and conjectures for semi-topological $K$-theory.} Math Annalen {\bf 330} (2004), 759--807.

\bibitem{FL} E.~Friedlander and H.~Lawson, {\em A theory of algebraic cocycles}. Ann. of Math. {\bf 136}(2), 361--428, 1992.

\bibitem{FW4} E.~Friedlander and M.~Walker, {\em Rational isomorphisms between $K$-theories and cohomology theories}. Inventiones Math. {\bf 154} (2003), 1--61.

\bibitem{FW3} \bysame, {\em Semi-topological $K$-theory of real varieties}. In Proceedings of the International Colloquium on Algebra, Arithmetic and Geometry, Mumbai 2000, Vol. {\bf I}, pages 219--326, 2002.

\bibitem{FW1} \bysame, {\em Comparing $K$-theories for complex varieties}. Amer. J. Math., {\bf 123}(5), 779--810, 2001.

\bibitem{Goodwillie} T.~Goodwillie, {\em Cyclic homology, derivations, and the free loopspace}. Topology {\bf 24} (1985), no.~2, 187--215.

\bibitem{Grayson} D.~Grayson, {\em The motivic spectral sequence}. Handbook of $K$-theory, volume 1, 39--69, 2005.

\bibitem{Grothendieck} A.~Grothendieck, {\em  On the de Rham cohomology of algebraic varieties}. Publ. Math. IH\'ES {\bf 29} (1966), 93--103.

\bibitem{HallNeemanRydh} J.~Hall, A.~Neeman and D.~Rydh. {\em One positive and two negative results for derived categories of
  algebraic stacks}. J. of the Institute of Math. of Jussieu {\bf 18} (2019), issue 5, 1087--1111.

\bibitem{Hesselink}  W.~Hesselink, {\em Concentration under actions of algebraic groups}. In: Paul Dubreil and Marie-Paule Malliavin Algebra Seminar, 33rd Year (Paris), Lect. Notes Math. {\bf 867} (1980), 55--89.

\bibitem{Keller} B.~Keller, {\em On differential graded categories}. International Congress of Mathematicians. Vol. {\bf II}, Eur. Math. Soc., Z\"urich, 2006, 151--190.

\bibitem{Kel99} \bysame, {\em On the cyclic homology of exact categories}. J. Pure Appl. Algebra {\bf 136} (1999), no. 1, 1--56.

\bibitem{Keller-Ann} \bysame, {\em Deriving DG categories}. Ann. Sci. Ecole Norm. Sup. (4) {\bf 27} (1994), no.~1, 63--102.

\bibitem{Krishna} A.~Krishna, {\em The completion problem for equivariant $K$-theory}. J. Reine Angew. Math. {\bf 740} (2018), 275--317. 

\bibitem{RR-K} \bysame, {\em Riemann-Roch for equivariant $K$-theory}. Advances in Mathematics {\bf 262} (2014), 126--192.

\bibitem{KR} A.~Krishna and C.~Ravi, {\em Algebraic $K$-theory of quotient stacks}. Annals of $K$-theory {\bf 3} (2018), no.~2, 207--233.

\bibitem{Lurie1} J.~Lurie, {\em Higher topos theory}. Annals of Mathematics Studies, {\bf 170} (2009).

\bibitem{Lurie2} \bysame, {\em Higher algebra}. Available at the webpage \url{https://www.math.ias.edu/~lurie/}.

\bibitem{Lurie} \bysame, {\em Proper morphisms, completions, and Grothendieck existence theorem}. Manuscript available at the following webpage \url{https://www.math.ias.edu/~lurie/}.

\bibitem{MR920520} A.~Magid. {\em Equivariant completions of rings with reductive group action}. J. Pure Appl. Algebra {\bf 49} (1987), 173--185.

\bibitem{MV} F.~Morel and V.~Voevodsky, {\em $\bbA^1$-homotopy theory of schemes}. Inst. Hautes \'Etudes Sci. Publ. Math. (1999), no. {\bf 90}, 45--143 (2001).

\bibitem{Neeman3} A.~Neeman, {\em The Grothendieck duality theorem via Bousfield's techniques and Brown representability}. J. Amer. Math. Soc. {\bf 9} (1996), no. 1, 205--236. 

\bibitem{Neeman-new} \bysame, {\em The connection between the $K$-theory localization theorem of Thomason, Trobaugh and Yao and the smashing subcategories of Bousfield and Ravenel}. Ann. Sci. \'Ecole Norm. Sup. (4) {\bf 25} (1992), no.~5, 547--566.

\bibitem{RO} A.~Rosenschon and P.~{\O}stvaer, {\em Descent for $K$-theories}. J. Pure Appl. Algebra {\bf 206} (2006), 141--152.

\bibitem{RO1} \bysame, {\em The homotopy limit problem for two-primary algebraic $K$-theory}. Topology {\bf 44} (2005) 1159--1179.

\bibitem{Schlichting} M.~Schlichting, {\em Negative $K$-theory of derived categories}. Math. Z. {\bf 253} (2006), no. 1, 97--134.

\bibitem{Segal} G.~Segal, {\em The representation ring of a compact Lie group}. Publications math\'ematiques de l'IH{\'E}S {\bf 34} (1968), 113--128.

\bibitem{Stacks} The stacks project authors, {\em The stacks project}. Available at the webpage \url{http://stacks.math.columbia.edu}.

\bibitem{Sumihiro1} H. Sumihiro, {\em Equivariant completion II}. J. Math. Kyoto Univ. {\bf 15} (1975), no.~3, 573--605.

\bibitem{Sumihiro2} \bysame, {\em Equivariant completion}. J. Math. Kyoto Univ. {\bf 14} (1974), no.~1, 1--28. 

\bibitem{Tab} G.~Tabuada, {\em $\bbA^1$-homotopy invariance of algebraic $K$-theory with coefficients and du Val singularities}. Annals of $K$-theory {\bf 2} (2017), no.~1, 1--25.

\bibitem{book} \bysame, {\em Noncommutative Motives}. With a preface by Yuri I. Manin. University Lecture Series, {\bf 63}. American Mathematical Society, Providence, RI, 2015.

\bibitem{A1homotopy} \bysame, {\em $A^1$-homotopy theory of noncommutative motives}. Journal of Noncommutative Geometry {\bf 9} (2015), no.~3, 851--875.

\bibitem{Concentration} G.~Tabuada and M. Van den Bergh, {\em Motivic concentration theorem}. Math. Research Letters {\bf 27} (2020), no.~2, 565--589.

\bibitem{Gysin} \bysame, {\em The Gysin triangle via localization and $\bbA^1$-homotopy invariance}. Trans. AMS {\bf 370} (2018), no.~1, 421--446. 

\bibitem{Thomason2} R.~Thomason, {\em Equivariant algebraic vs. topological K-homology Atiyah-Segal-style}. Duke Math. J. {\bf 56} (1988), 589--636.

\bibitem{Thomason-Adv} \bysame, {\em Equivariant resolution, linearization, and Hilbert's fourteenth problem over arbitrary base schemes}. Advances in Mathematics {\bf 65} (1987), no.~1, 16--34.

\bibitem{Thomason} \bysame, {\em Comparison of equivariant algebraic and topological K-theory}. Duke Math. J. {\bf 53} (1986), no.~3, 795--825. 

\bibitem{Thomason3} \bysame, {\em Lefschetz-Riemann-Roch Theorem and coherent trace formula}. Invent. Math. {\bf 5} (1986), 515--543.

\bibitem{Thomason-etale} \bysame, {\em Algebraic $K$-theory and \'etale cohomology}. Ann. Sci. Ec. Norm. Sup. {\bf 18} (1985), 437--552.

\bibitem{ThomasonActions} \bysame, {\em $K$-theory of group scheme actions}. Algebraic topology and algebraic $K$-theory (Princeton, N.J., 1983), 539--563,  Annals of Mathematics Studies, {\bf 113}.

\bibitem{TT} R.~Thomason and T.~Trobaugh, {\em Higher algebraic $K$-theory of schemes and of derived categories}. Grothendieck Festschrift III. Volume {\bf 88} of Progress in Math., 247--436, 1990.

\bibitem{Soule} C.~Soul\'e, {\em Operations on \'etale $K$-theory. Applications}. Algebraic $K$-Theory: Oberwolfach 1980. Springer Lecture Notes in Math. {\bf 966} (1982), 271--303.

\bibitem{Weibel} C.~Weibel, {\em Homotopy algebraic $K$-theory}. Algebraic $K$-theory and algebraic number theory (Honolulu, HI, 1987), Contemp. Math. {\bf 83}, AMS, Providence, RI, 1989, pp. 461--488.

\end{thebibliography}
\end{document}

\end{proof}